\documentclass{amsart}
\usepackage{url}
\usepackage[usenames,dvipsnames]{xcolor}
\usepackage[ruled,vlined]{algorithm2e}
\usepackage{mathtools}
\usepackage{subcaption}
\usepackage
[pagebackref=true, colorlinks=true, allcolors=blue, allbordercolors=blue,
        backref,
]{hyperref}
\usepackage{url,latexsym, amscd, amsfonts, eucal, mathrsfs, amsmath, amssymb, amsthm, rotating, graphicx}
\usepackage{tikz,enumitem,caption,stmaryrd,amsfonts,amssymb,systeme,comment,graphicx,colonequals,faktor,xfrac}
\usetikzlibrary{matrix,positioning,arrows,shapes,chains,calc,automata}
\usetikzlibrary{decorations.pathreplacing,cd}
\usepackage[asymmetric]{geometry}

\newtheorem{thm}{Theorem}[section]

\newtheorem{ass}{Assumption}
\newtheorem{prop}[thm]{Proposition}

\newtheorem{lemma}[thm]{Lemma}
\newtheorem{cor}[thm]{Corollary}

\theoremstyle{definition}

\newtheorem{defn}[thm]{Definition}

\theoremstyle{remark}

\newtheorem{rk}[thm]{Remark}
\newtheorem{ex}[thm]{Example}

\newcommand\Q{\mathbb{Q}}

\newcommand\Cp{\mathbb{C}_p}

\newcommand\C{\mathbb{C}}
\newcommand\F{\mathbb{F}}
\newcommand\G{\mathbb{G}}

\newcommand\Z{\mathbb{Z}}
\newcommand\B{\mathbb{B}}
\newcommand\T{\mathbb{T}}

\newcommand\m{\mathfrak{m}}

\renewcommand{\L}{L}

\renewcommand{\O}{\mathcal{O}}

\newcommand\inv{\mathop{\rm inv}\nolimits}

\renewcommand\O{\mathcal{O}}

\newcommand\ab{\mathop{\rm ab}\nolimits}
\newcommand\ord{\mathop{\rm ord}\nolimits}

\newcommand{\Hom}{\operatorname{Hom}}

\newcommand{\Div}{\operatorname{Div}}

\newcommand{\Red}{\operatorname{Red}}
\newcommand{\reg}{\operatorname{reg}}

\newcommand{\an}{\operatorname{an}}

\renewcommand{\div}{\operatorname{div}}

\newcommand{\dlog}{\operatorname{dlog}}
\newcommand{\Lie}{\operatorname{Lie}}

\newcommand{\pair}[1]{\langle #1 \rangle}

\newcommand{\GL}{\operatorname{GL}}

\newcommand{\PGL}{\operatorname{PGL}}

\newcommand{\BP}{{\mathbb P}}

\newcommand{\calC}{\mathcal{C}}
\newcommand{\calB}{\mathcal{B}}
\newcommand{\calE}{\mathcal{E}}
\newcommand{\calD}{\mathcal{D}}

\def\presuper#1#2%
{\mathop{}%
	\mathopen{\vphantom{#2}}^{#1}%
	\kern-12\scriptspace%
	#2}

\newcommand{\Ab}{{\operatorname{Ab}}}
\newcommand\Abint{\presuper\Ab\int}
\newcommand{\lAb}{{\operatorname{Ab\ }}}
\newcommand\lAbint{\presuper\lAb\int}

\newcommand*\angles[1]{\langle #1 \rangle}
\newcommand*\prnths[1]{\left( #1 \right)}

\begin{document}
\setcounter{secnumdepth}{3} 
\title[Algorithms for hyperelliptic Mumford Curves]{Algorithms for hyperelliptic Mumford Curves:  \\ $p$-adic Uniformization, $p$-adic integrals and $p$-adic heights}

\author[Kaya]{Enis Kaya}
\address{\hspace{-.2in} E. Kaya, Department of Mathematics, Bilkent University, 06800 Ankara, T\"urkiye}
\email{enis.kaya@bilkent.edu.tr}

\author[Masdeu]{Marc Masdeu}
\address{\hspace{-.2in} M. Masdeu,  Departament de Matem\`atiques\\Universitat Aut\`onoma de
	Barcelona\\08193 Bellaterra, Barcelona, Catalonia}
\email{marc.masdeu@uab.cat}

\author[M\"uller]{J. Steffen M\"uller}
\address{\hspace{-.2in} J. S. M\"uller,  Bernoulli Institute,
	Rijksuniversiteit Groningen,  Nijenborgh 9,  9747 AG Groningen, The Netherlands}
\email{steffen.muller@rug.nl}

\author[van der Put]{Marius van der Put}
\address{\hspace{-.2in} M. van der Put,  Bernoulli Institute,
	Rijksuniversiteit Groningen,  Nijenborgh 9,  9747 AG Groningen, The Netherlands}
\email{m.van.der.put@rug.nl}

\begin{abstract}
Mumford curves generalize the Tate uniformization of elliptic curves with split multiplicative reduction
and provide $p$-adic analogues of the uniformization of Riemann surfaces.
In this paper, we present several algorithms for hyperelliptic Mumford curves.
For a given hyperelliptic Mumford curve $X$ defined over a finite extension of the field of $p$-adic
numbers for some $p\neq 2$, we first describe how to compute a $p$-adic Schottky group $W$ that uniformizes $X$;
this is based on our extension to Kadziela's approximation theorem. As applications,
we explain how to use this uniformization in order to compute $p$-adic Abelian integrals and $p$-adic Schneider heights on $X$;
the latter uses Werner's formula expressing the $p$-part of the Schneider height
in terms of theta functions.
We illustrate our algorithms with numerical examples computed using the computer algebra system {\tt SageMath}. 
\end{abstract}

\maketitle

\tableofcontents

\section{Introduction}\label{S:}

In 1959, Tate proved that an elliptic curve over a $p$-adic 
field $K$ with split multiplicative reduction can be uniformized. His
foundational ideas gave rise to the field of rigid analytic geometry. Building on Tate's ideas,
Mumford~\cite{Mu72}
showed in 1972 that a curve $C/K$ of genus $g\ge 2$ with split degenerate
reduction can be uniformized by a group of transformations acting on the
$p$-adic upper half plane. More precisely, after passing to a finite
extension, if necessary, the rigid analytification of such a curve $C$ is isomorphic to
$\Omega/W$, where $\Omega\subset \BP^1(\C_p)$ and $W\subset \PGL_2(K)$ is a
so-called ($p$-adic) Schottky group. Moreover, Manin and
Drinfeld~\cite{MD73} used $p$-adic theta functions to show that the Jacobian $J$ of a Mumford curve admits a rigid-analytic uniformization by a 
$p$-adic torus $(\G_m^\times)^g/Q_W$, analogous to the complex uniformization of Jacobians as
complex tori. These functions can also be used to reformulate and expand on
Mumford's results, see for instance the book by Gerritzen and van der
Put~\cite{GP80} and the more recent work~\cite{vdPT26} of van der Put and
Top. We summarize the required background on Schottky groups,
Mumford curves and their Jacobians in \S\ref{S:Schottky}.

Both Tate's and Mumford's uniformizations have proved invaluable tools in
the arithmetic geometry of curves, but while there is a simple algorithm to
compute the Tate uniformization explicitly and this has been used
in algorithmic work on elliptic curves (see for
instance~\cite{SW13, mtt}), algorithmic
results on Mumford curves have so far been sparse, despite many possible
applications, including the computation of abelian integrals and $p$-adic
heights, of isogenies and of explicit examples for the tame inverse Galois
problem. In particular, the first two have the potential to form the basis
of model-free versions of the linear and quadratic Chabauty method to
compute rational points on curves.

The goal of the present paper is to remedy this situation and to take a major step toward making Mumford curves more algorithmically accessible.

Our goals are two-fold:
\begin{enumerate} 
  \item\label{GoalI} Algorithms to compute the $p$-adic uniformization starting
    with $C$;
  \item\label{GoalII} Algorithms for arithmetic applications starting with $W$.
\end{enumerate}

\subsection{Previous work}\label{}
We briefly summarize previous algorithmic work on Mumford curves. Starting
with an equation of a hyperelliptic Mumford curve, Teitelbaum~\cite{Tei88}
found formulas for the periods of the Jacobian when $g=2$. More
recently, Chow and
Jarvis~\cite{CJ23} found a $p$-adic version of the classical
arithmetic-geometric mean method to compute periods of genus~2 curves via
Richelot isogenies. 
Kadziela~\cite{Kad07} developed an exponential algorithm to compute the
Schottky group under certain conditions. Starting with the Schottky
group, Morrison and Ren~\cite{MR15} gave algorithms to compute various
objects such as a good fundamental domain, the periods of the Jacobian and
a canonical embedding of the curve. These relied on a naive algorithm
(see~\S\ref{subsec:naive}) for
computing $p$-adic theta functions. More recently, Masdeu and Xarles found
a fast iterative algorithm for this task (see~\cite{MX25} and also~\S\ref{subsec:iterative}).
\subsection{Algorithms starting with the curve}\label{}
We take Kadziela's work as a starting point for~\eqref{GoalI}. Kadziela uses an
approach due to van der Put~\cite{vdP78} who showed that a hyperelliptic
Mumford curve $C$ can be uniformized by a Schottky group $W$ (called a
Whittaker group) with the following property:
there is a certain discontinuous group $\Gamma$
which uniformizes $\BP^1$ in an appropriate sense such that $W$ is an
index~2 subgroup of $\Gamma$. This induces a
degree~2 cover of rigid analytic spaces $\Omega/W\to \Omega/\BP^1$ corresponding to
$C\to \BP^1$, for a suitable subdomain $\Omega=\Omega_W\subset
\BP^1(\C_p)$. More precisely, there is a $p$-adic theta function $F$
depending on $W$ and a set of generators of $\Gamma$ such that $F$ maps the
set $\T$ of fixed points of these
generators to the set $\B_C$ of branch points of $C\to \BP^1$.
More recently, van der Put's work has been revisited and extended by van der Put and
Top~\cite{vdPT26}. 
Based on van der Put's observation, Kadziela developed a digit-by-digit algorithm that computes
$\T$ from $\B_C$ under fairly restrictive conditions on the position of
$\T$.
The function $F$ is an infinite product, and Kadziela's
main results are a first-order approximation as well as an analysis of
the error term when truncating the product.
Note that the problem is more complicated than simply computing preimages
of $\B_C$ under $F$, since $F$ itself depends on $W$.

In \S\ref{S:FixedFromRam}, we extend Kadziela's approach in various ways: First, we relax his
conditions on the position of $\T$ in~\S\ref{S:Initial}, allowing us to deal with more general hyperelliptic
Mumford curves. Second, we present a new method to compute $\T$ from $\B_C$
in~\S\ref{S:Hensel}. It applies
a version of multivariate Hensel lifting that only requires an
approximation to the Jacobian matrix and yields linear convergence;
see Theorem~\ref{thm:hensel}.
Combined with the iterative algorithm of Masdeu and Xarles, our work
allows to compute $W$ in many previously inaccessible situations; it also
makes it possible to obtain much larger precision than Kadziela's original
approach, which is necessary for some of our intended applications.

Furthermore, motivated by our applications, we explain how to compute the $p$-adic
uniformization of points on $C$ in \S\ref{S:unif_th}: given $P\in
C(K)$, we compute $z\in
\Omega(K)$ that maps to $P$. This requires an explicit description of the
field of meromorphic functions on $\Omega/W$; see also~\cite{vdPT26}.

\subsection{Algorithms starting with the Schottky group}\label{}
We now turn to~\eqref{GoalII}.
Let $J$ be the Jacobian of a Mumford curve $C/K$.
\subsubsection{Abelian integrals}\label{}
The abelian integral on $J$ is defined in terms of the
abelian logarithm $\log_J$ on the $p$-adic Lie group $J(K)$ (see~\cite{Zar96}): The logarithm of a
point in $J(K)$ is a linear function on the holomorphic differentials
on $J$, and the abelian integral between $P,Q\in J(K)$ of such a
differential $\omega$ is
\[
  \presuper\Ab\int_P^Q\omega = \log(Q)(\omega)-\log(P)(\omega) \in K\,.
\]
One can then pull back the logarithm and the integral to the curve.
On Mumford curves, the holomorphic differentials are generated by $\dlog u$
for certain $p$-adic theta functions $u$, so that the abelian logarithm 
essentially boils down to evaluating $\log u$. However, this is not
well-defined on $J(K)\simeq (K^\times)^g/Q_W$, where $Q_W$ is the period
matrix of $J$, and we have to correct by a term
that ensures periodicity with respect to the multiplicative lattice spanned
by $Q_W$, which we construct from a naive higher-dimensional generalization of the
$\mathcal{L}$-invariant of a Tate curve. See \S\ref{S:abelian} for
details.

A different algorithm to compute abelian integrals on hyperelliptic curves
is due to Katz and Kaya (see~\cite{KK22}), and we show that their algorithm
gives the same result as ours in an Example~\ref{E:AbEx}.
The main motivation for abelian integration is the method of Chabauty and
Coleman (see~\cite{MP12}) to compute the rational points on curves defined over number fields. This method has been
quite successful in practice, but one of its drawbacks is that it
requires a model of the curve, even when there is a natural uniformization
or modular interpretation available.
While our algorithm is restricted to Mumford curves and their Jacobians, it
has the advantage (compared to~\cite{KK22} as well as algorithms for
$p$-adic integration on curves of good reduction) that it only depends on the $p$-adic uniformization. This makes
it a natural candidate for a Chabauty--Coleman method that does not
require a model of the curve.

\subsubsection{$p$-adic heights}\label{}
Chabauty--Coleman requires the rank of the Jacobian to be less than the
genus. Using techniques from non-abelian $p$-adic Hodge theory,
Kim~\cite{kim05:motivic_fundamental_group,kim09:unipotent_albanese} has proposed an
ambitious research program to remove this condition, still using $p$-adic
integrals. Chabauty--Kim has recently been made explicit under certain
conditions, for instance when the
rank equals the genus and the Jacobian has Picard number $>1$
(see~\cite{BD18}), and has been used to compute the rational points on various
modular curves of arithmetic interest (see for instance \cite{BDMTV19,
BDMTV23}).
The main tool was a reformulation of Chabauty--Kim in terms of $p$-adic
height pairings, which is possible under the conditions considered
in~\cite{BD18}. 

Motivated by this and other applications, we develop a new algorithm to
compute $p$-adic heights on Jacobians of Mumford curves over a number field
$F$, see \S\ref{S:heights}.
There are different constructions of $p$-adic heights, but they can all be written as a sum of local
terms, one for each finite place of $F$. 
The local components away from $p$ are classical and easily described
using arithmetic intersection theory, but the local components above $p$
are more tricky. 
Our algorithm is based on a 
reformulation due to Werner~\cite{Wer96} of a construction of $p$-adic
heights due to Schneider~\cite{Sch82}~\footnote{While in good reduction,
all known constructions of $p$-adic heights are essentially equivalent,
in bad reduction the Schneider height is known to differ from other
constructions.}. 
All previous versions of algorithms for local $p$-adic height pairings
above $p$ required a model of $C$ (see~\cite{BBM17, BDMTV23, GM25, BKM25}),
which makes them difficult to apply for large genus. In contrast, Werner's
formula, and hence our algorithm, work
  directly on the rigid uniformization of the curve. We hope that this can
  be used to develop a model-free version of the quadratic Chabauty method
  in the future. 

  We expect that our algorithms for abelian integrals and $p$-adic heights
  will be particularly useful for modular curves, which often have split
  degenerate reduction at primes dividing the level, as well 
  as Shimura curves, which are often covered by Mumford curves;
  see~\cite{AM19}.

  We also note that another possible application of our algorithm for
  $p$-adic heights is to gather numerical evidence for a
  yet-to-be-formulated higher dimensional version of the $p$-adic version of the conjecture of Birch and
Swinnerton--Dyer for elliptic curves with split multiplicative reduction at
$p$ due to Mazur--Tate--Teitelbaum; see~\cite{mtt, BMS16}.

\subsubsection{Other applications}\label{}
 We mention two other possible applications: As shown by
 Kadziela~\cite{KadPaper}, we can find isogenies between the Jacobians of
 two Mumford curves  by finding $\Z$-linear relations between the
 logarithms of their period
 lattices, similar to the complex setting. Moreover, Bisatt and
 Dokchitser~\cite{BD21} use Mumford curves to prove that for every
 squarefree integer $N$ and every $g>0$, there is a Jacobian of dimension
 $g$ having tame mod $N$ Galois representations, and they apply this result to show that
 for all primes $p$ (satisfying a condition that is believed to always
 hold), there is a solution to the tame inverse
Galois problem for $\mathrm{GSP}_{2g}(\F_p)$, though they do not give
explicit examples.
Our algorithms should make it possible to find 
explicit solutions to this problem.

\subsection{Implementation and dependencies}\label{S:Imp}
We have implemented our algorithm in the computer algebra system~{\tt
SageMath}~\cite{sagemath}. The implementation can be found
at~{\url{https://github.com/mmasdeu/hyperellipticmumford}}.
Various examples of applications of our code are presented in
\S\ref{sec:examples}.

We rely on an
implementation of an algorithm  due to Morrison--Ren to move points into a
good fundamental domain of a Schottky group available from
{\url{https://arxiv.org/src/1309.5243}}. The implementation of the
algorithm of Masdeu and Xarles to compute $p$-adic theta functions is part
of the~{\tt darmonpoints} package, available at
{\url{https://github.com/mmasdeu/darmonpoints}}; our code also relies on various
other components of that package.
We use~{\tt Magma}~\cite{Magma} to compute local heights away from~$p$; for
Example~\ref{x039}
we need the code at {\url{https://github.com/emresertoz/neron-tate}}. 

\subsection{Generalizations}\label{}
All our results should admit extensions to superelliptic Mumford curves.
The theory of the $p$-adic uniformization of such curves has been developed by van Steen~\cite{VanSteen}
and Yelton~\cite{Yel23, Yel24}.


\subsection{Notation}\label{S:Not}

	In this article $K = (K,|\cdot|)$ denotes a finite extension of $\Q_p$,
where $p>2$, with valuation
ring $\O$, maximal ideal $\m$ and residue field $k\simeq \F_{p^n}$.
Let $\pi$ be a uniformizer of $\O$ and let $v$ be the discrete
valuation on $K$, normalized such that $v(\pi) = 1$. Throughout, we fix a
choice for $\infty\in\BP^1(K)$ and embeddings $K\hookrightarrow
\overline{K}\hookrightarrow \C_p$
 where $\overline{K}$ is an algebraic closure of $K$..

For us, an~\textit{open ball} in $\BP^1(\C_p)$ will be a subset
$$B(c,r)\colonequals \{z\in \BP^1(\C_p)\,:\, |z-c|<r\}\,;$$ the
corresponding~\textit{closed ball} is
$$B^+(c,r)\colonequals \{z\in \BP^1(\C_p)\,:\, |z-c|\le r\}\,.$$
We write $B^+$ for the closed ball corresponding to an open ball $B$.
Of course, both open and closed balls are both open and closed in the
$p$-adic topology.

A curve $C$ over a field $K$ is \textit{nice} if it is smooth, projective and geometrically integral.

\subsection*{Acknowledgements} 
We are grateful for useful discussions with
Alex Best, Raymond van Bommel, Tim Dokchitser, Timo Keller, Drew Sutherland, Jaap Top, Jan Vonk and
Xavier Xarles. Some of the work in this article was carried out in the
Nesin Mathematics Village and Centro de Ciencias de Benasque Pedro Pascual.
We would like to thank these institutes for the generous hospitality and
welcoming environment. E.K. was supported by NWO grant 613.009.124, 
FWO grant GYN-D9843-G0B1721N and TÜBITAK 2232-B fellowship 124C816 during various stages of this article. J.S.M. was
supported by NWO grant VI.Vidi.192.106.

\section{Positions in $\BP^1$}\label{S:position}
Much of this article is concerned with subgroups of $\PGL_2$ acting on
$\BP^1$. We will use the following notation.
\begin{defn}\label{D:}
Let $\Gamma\subset \PGL_2(K)$ be a subgroup.
  The set of~\textit{limit points} of $\Gamma$ is  the set of all points
  $z\in \BP^1(\C_p)$ such that there is an infinite sequence
  $(\gamma_n)$ of distinct elements of $\Gamma$ and $x\in \BP^1(\C_p)$ such that
  $\lim_{n\to\infty}\gamma_n(x)=z$. A point $z\in \BP^1(\C_p)$ is
  an~\textit{ordinary point} of $\Gamma$ if it is not a limit point. We
  call $\Gamma$~\textit{discontinuous} if it has ordinary points
  \footnote{This
  is enough, because $K$ is locally compact.}. We write
  $\Omega_\Gamma$ for the set of ordinary points of a discontinuous group
  $\Gamma$.
\end{defn}
We will always assume for every discontinuous group  that $\infty$
is not a limit point.
Throughout this section, let $g\ge 1$ and consider a subset $\T=\{a_0,b_0,\dots,
a_g,b_g\}\subset\BP^1(K)$ of size $2g+2$.

\begin{defn}\label{elts:gens_in_good_pos}
  For $i\in \{0,\ldots,g\}$, let  $s_i\in \mathrm{PGL}_2(K)$ be the unique
  matrix of order $2$ with fixed points $\{a_i,b_i\}$. 
  We say that  $\T$ is in {\it good position} if the group
  $$\Gamma(\T)\colonequals \pair{s_0,\dots,s_g}\le \PGL_2(K)$$
  is discontinuous and equal to the free product $\angles{s_0}*\cdots*\angles{s_g}$.
\end{defn}
Unfortunately, no useful {\it necessary and sufficient} condition for good position seems to be known.

\subsection{A necessary condition for good position}

We first describe a necessary condition, see Lemma~\ref{L:necescond}.
Let $S\subset \BP^1(K)$ be a finite set of size $\ge 3$, and set $$S^{(3)}
\colonequals S\times S \times S - 
\Delta \colonequals \{(s_1,s_2,s_3)\mid s_1 = s_2 \text{ or } s_1 = s_3 \text{ or }
s_2 = s_3\}.$$

\begin{defn}
For every $s = (s_0,s_1,s_2)\in S^{(3)}$ denote by $\gamma_{s}$ the unique automorphism
such that $\gamma_s(s_i) = i$ for all $i$. Composing this with the canonical reduction map $\Red \colon \BP^1(K) \to \BP^1(k)$ gives a surjective map
\[\Red_{s} \colon \BP^1(K) \to \BP^1(k).\]
Moreover, for each $s\in S^{(3)}$, the set
\[\BP_s \colonequals \big\{\Red_{s}^{-1}(t)\cap S \mid t \in \Red_{s}(S)\big\}\]
is a partition of $S$. The relation on $S^{(3)}$ defined by
\[s \sim s^\prime \ \ \ \Longleftrightarrow \ \ \ \BP_s = \BP_{s^\prime}\]
is an equivalence relation; let us denote the equivalence class of an element $s\in S^{(3)}$ by $[s]$. Consider the following graph:
\begin{itemize}
	\item vertices correspond to the classes $[s]$ for $s\in S^{(3)}$, and
	\item two vertices $[s],[s^\prime]$ are connected by an edge if and only
    if one can write $\BP_s = \{U_1,\dots,U_r\}$, $\BP_{s^\prime} = \{V_1,\dots,V_{r^\prime}\}$ with $U_1 = V_2\cup \cdots \cup V_{r^\prime}$, $V_1 = U_2\cup \cdots \cup U_{r}$.
\end{itemize}
It is a finite tree, called the {\em tree} of $S$ and denoted by $T(S)$.
\end{defn}

\begin{defn}
Write $S^{(3)} / \sim \ = \{[s_1],\dots,[s_n]\}$. The {\em reduction of $\BP^1$ with respect to $S$} is defined as
\[\Red_S\colon \BP^1(K) \to \BP^1(K)^n\to \BP^1(k)^n,\]
where the first map is
$\gamma_{s_1}\times\cdots\times \gamma_{s_n}$ and the second map is the canonical reduction.
\end{defn}

Define $Z = Z_S$ as the image $\Red_S(\BP^1(K))$. Then
\begin{enumerate}
	\item $Z$ is a union of $n$ lines $L_1,\dots,L_n$, two of which intersect in at most one ordinary double point, 
	\item the tree $T(S)$ is isomorphic to the intersection graph of the $L_i$'s,
	\item the restriction of $\Red_S$ to $S$ is injective, and each point on $\Red_S(S)$ lies on just one $L_i$.
\end{enumerate}

\begin{defn}
Let $L$ be a component of $Z$, let $[s]$ be the corresponding class in $S^{(3)} / \sim \ $, and take two distinct points $x,y\in S$. We say that $L$ {\em separates} $x,y$ if the partition $\BP_s$ contains the singleton sets $\{x\}$ and $\{y\}$.

\end{defn}
Equivalently, $L$ separates $x,y$ if and only if
\begin{itemize}
	\item $\Red_s(z) \neq \Red_s(x)$ for all $z\in S\setminus \{x\}$, and 
	\item $\Red_s(z) \neq \Red_s(y)$ for all $z\in S\setminus \{y\}$.
\end{itemize}
This condition is quite easy to check  algorithmically.

We get the following necessary condition for good position:

\begin{lemma}\cite[Chapter~IX, \S2.5.1]{GP80}\label{L:necescond}
  Suppose that $\T$ is in good position. Then every line $L$ in $\Red_\T(\BP^1(K))$ separates at most one pair $(a_j,b_j)$.
\end{lemma}

Other necessary conditions are discussed in~\cite[\S7]{vdPT26}.

\subsection{Sufficient conditions}

We now discuss sufficient conditions  for a set to be in good position.

\begin{defn}\cite[\S7.1.1]{vdPT26}
  We say that the set $\T$ is in {\em closed disk position} 
  if the reduction $R$ of $\mathbb{P}^1$ with respect to $\{a_0,a_1,\dots ,a_g\}$ 
     has the property that the image $\bar{b}_j$ of each $b_j$ in $R$ coincides with the image
     $\bar{a}_j$ of $a_j$. In that case, the reduction $\Red_\T$ is obtained
     from $R$ by  replacing each $\bar{a}_j$ by an intersecting new line
     containing the images
     of $a_j$ and $b_j$. 
\end{defn}

For $j=0,\dots ,g$, let $\calB_j$ ($\calB^+_j$, respectively), denote the smallest open (closed, respectively) disk containing
$\{a_j,b_j\}$.   
   
\begin{lemma}\label{L:ballgood}\cite[\S1, Proposition]{vdP78}
  If $\T$ is in closed disk position, then $\T$ is in
  good position. Moreover, $\BP^1\setminus \calB_0\cup\cdots\cup \calB_g$ is a
  fundamental domain for $\Gamma(\T)$ in the sense of~\cite[\S1]{vdP78}.
\end{lemma}

\begin{rk}\label{R:good}
  \hfill
  \begin{enumerate}
    \item Suppose that the image of $\infty$ is an
ordinary point with respect to $\Gamma$
and not on any of the lines connecting $\bar{a}_j,\bar{b}_j$.
Then $\T$ is in closed disk position if and only if 
\begin{equation}\label{}
       \calB^+_i\cap \calB^+_j=\emptyset\quad\text{ for all}\; i\neq j\,.
\end{equation}
    \item 
      There is a shorter way to formulate the closed disk condition: 
            The configuration $\Red_\T$ has end lines $L_0,\dots ,L_g$. Each $L_j$ has one point of intersection with the rest 
            of the configuration and contains the distinct images of $a_j$ and $b_j$.
    \item 
      In~\cite[Chapter~IX, \S2.5.2]{GP80} and~\cite[p.~76]{VanSteen}, other sufficient conditions  for  good position  are discussed. 
    \item In~\cite[\S5]{vdPT26}, van der Put and Top introduce a notion of~\textit{restricted}
      position and show that restricted implies good position
      (see~\cite[Theorem 5.4]{vdPT26}). In particular, the closed disk
      position is restricted (see~\cite[\S5.1.1]{vdPT26}).
  \end{enumerate}
\end{rk}

We would like to check algorithmically whether the set $\T$ is in closed
disk
position, and to generalize this position. We use an approach due to Kadziela~\cite{Kad07}. Applying a
transformation, if necessary, we may assume that $0,1,\infty\in \T$.
Kadziela showed that we may assume something stronger.

\begin{defn} \label{D:KadzPos}
We say that $\T$ is in {\em Kadziela position} if  
\begin{enumerate}
  \item $a_0 = 0$, $a_g=1$, $b_{g} = \infty$\,, and
  \item $0<|b_0|<|a_1|\le |b_1|\le\cdots\le |b_{g-1}|<1$\,.
\end{enumerate}
\end{defn}

By applying the following lemma, we may assume that $\T$ is in Kadziela position.

\begin{lemma}{\cite[Proposition~5.2]{Kad07}}\label{L:niceRos} 
  Let $S \subset \mathbb{P}^1(K)$ be any set for which $\#S = 2g + 2\geq 3$ 
  such that every line in $\Red_S(\mathbb{P}^1(K))$ separates at most two points.
  Then there exists $\sigma\in \PGL_2(K)$ such that
  $\sigma(S)$ is in Kadziela position.
\end{lemma}

\begin{proof}
	We only sketch the construction of the required transformation $\sigma$.
  First, we may assume without loss of generality that $S =
  \{0,y_1,\dots,y_{2g-1},1,\infty\}$ with $|y_i|\leq 1$ for all $i$. Then
  there is a line in $\Red_S(\mathbb{P}^1(K))$ that separates some $x,y\in S$ with $x,y\neq \infty$. Set
	\[\sigma_1 = 
	\begin{pmatrix}
		0 & x-y \\
		1 & -y
	\end{pmatrix}\]
	which transforms $x,y$ to $1,\infty$, respectively. There is a line in
  $\Red_{\sigma_1(S)}(\mathbb{P}^1(K))$ that separates some $u,v\in \sigma_1(S)$ with $\{u,v\}\cap \{1,\infty\} = \emptyset$. Set
	\[\sigma_2 = 
	\begin{pmatrix}
		0 & -v \\
		1 & 1-v
	\end{pmatrix}\]
	which transforms $v,1,\infty$ to $0,1,\infty$, respectively. Finally, the map $\sigma = \sigma_2\sigma_1$ has the desired property.	
\end{proof}

For a set $S$ as in Lemma~\ref{L:niceRos}, determining an explicit transformation which maps $S$ to a set that is in Kadziela position is easy. First of all, there are precisely $(2g+2)\times (2g+1)\times (2g) $ transformations $\sigma$ with the property that $\{0,1,\infty\}\subset \sigma(S)$, and for each one of these, one can check if $\sigma(S)$ satisfy the second condition in Definition~\ref{D:KadzPos}. When $\#S = 2g + 2$ is small, for instance when $g=2$, this brute-force method works quite well, and this is what we do in practice. In general, one can make use of the following algorithm whose recipe is based on the proof of Lemma~\ref{L:niceRos}:

\begin{algorithm}
	\caption{\bf Getting Kadziela position}
	\label{A:Trans_Kad_Pos}
  \KwIn{A subset $S \subset \mathbb{P}^1(K)$ with $2g + 2\geq 3$ elements
  such that every line in $\Red_S(\mathbb{P}^1(K))$ separates at most two points.}
	\KwOut{A linear fractional transformation $\sigma\in \PGL_2(K)$ with the property that
    $\sigma(S)$ is in Kadziela position. }
	
	\begin{enumerate}
		\item Find $\sigma_0\in \PGL_2(K)$ such that
		\[\sigma_0(S) = \{0,y_1,\dots,y_{2g-1},1,\infty\},\ \ \ |y_i|\leq 1 \text{ for all } i.\]
		\item Find a finite pair $\{x,y\}\subset \sigma_0(S)$ that is separated by a line in
      $\Red_{\sigma_0(S)}(\mathbb{P}^1(K))$, 
      and form the map $\sigma_1 = 
		\begin{pmatrix}
			0 & x-y \\
			1 & -y
		\end{pmatrix}$.
		\item Find a pair $\{u,v\}\subset \sigma_1\sigma_0(S)$ with $\{u,v\}\cap \{1,\infty\} = \emptyset$
      that is separated by a line in
      $\Red_{\sigma_1\sigma_0(S)}(\mathbb{P}^1(K))$, and form the map $\sigma_2 = 
		\begin{pmatrix}
			0 & -v \\
			1 & 1-v
		\end{pmatrix}$.
		\item Return $\sigma = \sigma_2\sigma_1\sigma_0$.
	\end{enumerate} 
\end{algorithm}

\begin{cor}\label{C:goodKad}
	If $\T$ is in good position, then
	there exists $\sigma\in \PGL_2(K)$ such that
	$\sigma(\T)$ is in Kadziela position.
\end{cor}
\begin{proof}
	By Lemma~\ref{L:necescond}, a set in good
	position satisfies the condition of Lemma~\ref{L:niceRos}.
\end{proof}

The unique matrices $s_i\in \PGL_2(K)$ of order 2 fixing $a_i$ and $b_i$ are as follows:
\[s_0 = 
\begin{pmatrix}
b_0 & 0 \\
2 & -b_0
\end{pmatrix}, \text{ fixing } \{0,b_0\}; \ \ \ \ \ \ \ 
s_g = 
\begin{pmatrix}
	1 & -2 \\
	0 & -1
\end{pmatrix}, \text{ fixing } \{1,\infty\};
\]
\[s_i = 
\begin{pmatrix}
a_i+b_i & -2a_ib_i \\
2 & -(a_i+b_i)
\end{pmatrix}, \text{ fixing } \{a_i,b_i\},\	 i=1,\dots,g-1.\]
For $i=0,\dots,g-1$, the open disk $\calB_i$ has center $c_i$ and radius $r_i$,
where
\begin{equation}\label{rici}
c_i = 
\begin{cases} 
b_0/2 & \text{if } i = 0, \\
(a_i+b_i)/2   & \text{if } i = 1,\dots,g-1, 
\end{cases} \ \ \ 
r_i = 
\begin{cases} 
b_0/2 & \text{if } i = 0, \\
(a_i-b_i)/2   & \text{if } i = 1,\dots,g-1. 
\end{cases}
\end{equation}
Following Kadziela, we define
\begin{equation}\label{D:dij}
d_{ij} \colonequals \frac{|r_i|}{|c_i-c_j|},\ \ \ i,j\in\{0,\dots,g-1\}
  \text{ with } i\neq j\,.
\end{equation}
\begin{defn}\label{D:}
  We say that $\T$ is in~\textit{strong Kadziela position} if $\T$ is in Kadziela position and 
  we have $d_{ij}<1$ for all  $i\neq j$ in
  $\{0,\ldots,g-1\}$. 
\end{defn}
 The condition $d_{ij}<1$ for all $i\ne j$ means that the set 
$\T'=\{0,b_0/p,a_1/p,\ldots,1/p,\infty\}$ is in closed disk
 position. 
Since $\T$ is in good position if and only if $\T'$ is, Lemma~\ref{L:ballgood}
implies:
\begin{cor}\label{P:Kad} \cite[Theorems~5.3-5.7]{Kad07} 
  Suppose that $\T$ is in strong Kadziela position.	Then $\T$ is in good position.		
\end{cor}

Kadziela's main approximation theorem~\cite[Theorem 6.10]{Kad07} and the
resulting algorithm (outlined in~\cite[\S6.2]{Kad07}) assume that $\T$ is in
strong Kadziela position. We will generalize this below in~\S\ref{S:Initial} by only
assuming the following weaker condition:

\begin{defn}\label{D:}
  We say that $\T$ is in~\textit{weak Kadziela position}
   if $\T$ is in Kadziela position and 
  we have $d_{ij} \leq 1$ and $d_{ij}d_{ji} < 1$ for all  $i\neq j$ in
  $\{0,\ldots,g-1\}$. 
\end{defn}
The first condition means that $\calB_i\cap \calB_j= \varnothing$ for all $i\ne j$. 
In contrast to strong Kadziela position, not every set in weak Kadziela
position is also in good position, see~\cite[\S7]{vdPT26}. 
\begin{rk}\label{R:}
It is obvious that strong Kadziela implies weak Kadziela and that there are sets that are in
  Kadziela position, but are not in weak Kadziela position.
\end{rk}

\subsection{Position in genus~2}\label{S:posg2}
Suppose that $g=2$ and that $\T=\{0,b_0,a_1,b_1,1,\infty\}$ is in good
position. Then by~\cite[Chapter~IX, (2.5.3)]{GP80} and \cite[\S7.2]{vdPT26}
there are three different types of position, called (a), (b) and (c), where (a) is the closed
disk position. Suppose that $\T$ is in
Kadziela position. Then (a) is precisely the strong Kadziela position.
By~\cite[Remark~8.4]{vdPT26}, $\T$ is of type (b) if it is not in strong
Kadziela position and satisfies
\begin{equation}\label{caseb}
  0 < |b_0 | < |b_1 | = |a_1 | < 1\,.
\end{equation}
Similarly,~\cite[Observation 8.6]{vdPT26} says that $\T$ is of type (c) if
it is not in strong Kadziela position and satisfies
\begin{equation}\label{casec}
0 < |b_0 | < |b_1 | < |a_1 | < 1\,.
\end{equation}

\begin{prop}{\cite[Proposition~7.3]{vdPT26}}
The set $\T$ is in good position if it's of type (a) or (c). It might 
  or might not be in good position if it's of type (b).
\end{prop}
\section{Schottky groups, theta functions and Mumford
curves}\label{S:Schottky} 
 
Mumford showed that every split degenerate curve of genus $g\ge 2$ over $K$ can be
uniformized by a Schottky group in $\PGL_2(K)$. We recall the theory of
Schottky groups, their theta functions and Mumford curves here. Our main reference
is~\cite{GP80}, but we   will use the notation of~\cite{MX25}.

\subsection{Schottky groups}\label{}

We will be mostly, but not exclusively, concerned with the following
special discontinuous groups.
\begin{defn}\label{D:Schottky}
  We call a subgroup $W\subset \PGL_2(K)$ a \textit{Schottky group} if $W$
  is discontinuous, finitely generated and torsion-free. 
\end{defn}
\begin{lemma}\label{L:DiscSchott}
  Let $\Gamma\subset \PGL_2(K)$ be a finitely generated discontinuous
  group. Then
  \begin{enumerate}
\item $\Gamma$ has a finite index normal Schottky subgroup $W$, and we have
  $\Omega_W = \Omega_\Gamma$;
    \item $\Gamma$ is Schottky itself if and only if it is
  discrete and free. 
  \end{enumerate}
\end{lemma}
\begin{proof}
  The first statement is~\cite[Chapter~1,~Theorem~3.1~(1)]{GP80}. The second
  follows from~\cite[Chapter~I,~(1.1.6) and Theorem~3.1~(2)]{GP80}.
\end{proof}

\begin{defn}\label{D:GoodFundDom}
	Let $W\subset \PGL_2(K)$ be a Schottky group of rank $g\ge 2$. A
  \textit{good fundamental domain} for $W$ is a set $\mathcal{F} \subset
  \mathbb{P}^1(\C_p)$ with the following properties:
      \begin{enumerate}
        \item\label{Good1} $\mathcal{F}$ is the complement of open balls $B_i$, where
          $i\in \{\pm1,\ldots,\pm g\}$,
          with centers in $K$ and radii in $|\overline{K}^\ast|$;
    \item\label{Good2} the corresponding closed balls $B_i^+$ are mutually disjoint, and
    \item\label{Good3} $W$ has generators $w_1,\dots,w_g$ that satisfy
      \[w_i(\mathbb{P}^1_K\setminus B_{-i}) = B_{i}^+ \ \ \ \text{and} \ \
          \  w_i(\mathbb{P}^1_K\setminus B_{-i}^+) = B_{i},\ \ \ i = 1,\dots,g.\]
          In this case, we call $w_1,\ldots,w_g$~\textit{good
          generators}.
      \end{enumerate}
\end{defn}

We now collect the facts that we need about good fundamental domains. For
proofs, see \cite[Chapter~I, (4.1.3) \& (4.1.4)]{GP80}.
\begin{prop}\hfill 
  \begin{enumerate}
    \item 
  Let $W\subset \PGL_2(K)$ be a Schottky group. Then there is a good fundamental domain for $W$, and 
      for every good fundamental domain $\mathcal{F}$ for $W$, we have $\bigcup_{w\in
      W}w \mathcal{F} = \Omega_W$. In particular, $\mathcal{F}\subset \Omega_W$.
    \item Conversely, if $\mathcal{F}\subset \BP^1_K$ satisfies properties~\eqref{Good1}
      and~\eqref{Good2} of Definition~\ref{D:GoodFundDom}, then there are
      $w_1,\ldots,w_{g}\in\PGL_2(K)$ such that
      $\langle w_1,\ldots,w_{g}\rangle$ is Schottky and
      satisfies~\eqref{Good3}.
  \end{enumerate}
	\end{prop}

Thanks to this result, for a given point $P \in \Omega_W$, there is
a point $Q\in\mathcal{F}$ such that $P$ and $Q$ are equivalent modulo the action
of $W$; in other words, ordinary points can be moved to the fundamental domain. This has been made algorithmic by Morrison and Ren; see~\cite[Subroutine~2.6]{MR15}. They also give an algorithm to compute
a good fundamental domain and generators given any set of generators of $W$ (see~\cite[Algorithm~4.8]{MR15}), and we may and will therefore assume that whenever we deal with a Schottky group, we have good generators available.
 
\begin{rk}\hfill
  \begin{enumerate}
    \item Not every set of generators of a Schottky group is good for some fundamental domain.
    \item If a set of generators of a Schottky group is good, then it might be good for more than one fundamental domain, see~\cite[\S5.2]{MX25}.
  \end{enumerate}
\end{rk}

\subsection{Theta functions}\label{S:Theta}
Gerritzen and van der Put used non-archimedean theta functions on Schottky
groups to make 
Mumford's theory of $p$-adic uniformization of split degenerate curves
explicit. For a  Schottky group $W\subset \PGL_2(K)$ they define, 
in~\cite[Chapter~II]{GP80}, the  theta function $\theta_W$ by
  $$\theta_{W}(a,b;z) \colonequals \prod_{w\in W}\frac{z-w a }{z-w b}$$
  for $a,b,z\in \Omega_W$. 
  However, when using this definition, one needs to deal with various special cases and
  several statements and proofs below become fairly cumbersome. We
  instead use a less coordinate-dependent  definition from~\cite[\S3]{MX25}. 
  First we recall, for distinct $z,y,a,b\in K$, the \textit{cross ratio}
  \begin{equation}\label{CrossRatio}
  (z,y;a,b)  \colonequals \frac{z-a}{z-b}\cdot\frac{y-b}{y-a}\,.
  \end{equation}
  We extend this definition to $\BP^1(K)$ by defining
  $$
  (z,y;z,b) \colonequals (z,y;a,y) \colonequals 0,\;\;
  (z,z;a,b) \colonequals (z,y;a,a) \colonequals 1,\;\;
  (z,y;a,z) \colonequals (z,y;y,b) \colonequals \infty
  $$
and requiring  for all $\gamma \in \PGL_2(K)$:
$$ (\gamma z, \gamma y; \gamma a, \gamma b)=(z,y;a,b)$$
We further extend the cross ratio to a pairing $(D,E)$ for $D,E \in Z^0(\BP^1(K))$;
  the group of zero cycles on $\BP^1(K)$ of degree~0. 
\begin{lemma}{\cite[Proposition~3.1]{MX25}}\label{L:thetaprops}
  The pairing $(\cdot,\cdot)$ is
  bilinear, symmetric and we have 
   $(\gamma D,\gamma E) = (D,E)$ for all $\gamma\in \PGL_2(K)$ and $D,E \in
   Z^0(\BP^1(K))$.
\end{lemma}
We now use this pairing to define the theta function with respect to a
Schottky group $W$ of rank~$g\ge 2$. We fix good generators $w_1,\ldots,w_g$ and
we denote for a positive integer $n$ by $W_n$ the set of reduced words of
length at most $n$ in $w_1,\ldots,w_g$.
\begin{defn}\label{D:thetaW}
  The \textit{theta function (or theta pairing) with respect to $W$} is the pairing on
  $Z^0(\Omega_W(K))$ defined by
  \begin{equation}\label{DEW}
    (D,E)_W\colonequals \lim_{n\to \infty}\prod_{w\in W_n}(D,w E)\,.
  \end{equation}
\end{defn}
The limit exists by~\cite[IV.1]{GP80}.
As a special case, we recover the original definition of Gerritzen and van
der Put, namely
$$
\theta_W(a,b;z) = (a-b,z-\infty)_W\,.
$$
We collect a few useful properties of the theta function
from~\cite[\S3]{MX25}.
\begin{lemma}
  Let $Z^0(\Omega_W(K))_W \colonequals H_0(W,Z^0(\Omega_W(K)))$ be the
  coinvariant zero-cycles on $\Omega_W(K)$.  Then
  \begin{enumerate}
    \item $\iota(w)\colonequals z_0-w z_0$ defines a homomorphism $\iota\colon W\to
      Z^0(\Omega_W(K))_W$ that does not depend on $z_0$;
    \item  $(\cdot,\cdot)_W$ extends to a pairing
      on $Z^0(\Omega_W(K))_W$ that satisfies the properties of
      Lemma~\ref{L:thetaprops}. 
  \end{enumerate}
  \end{lemma}
  
  \begin{cor}\label{C:}
   For all $D,E \in Z^0(\Omega_W(K))$, we have $(\iota D, \iota E)_W = (D,E)_W$.
  \end{cor}
  
\begin{lemma}\label{L:thetaconj}
  If $\gamma\in\PGL_2(K)$ is in the normalizer of $W$, then $(\gamma D,\gamma E)_W = (D,E)_W$ for all $D,E \in Z^0(\Omega_W(K))$.
\end{lemma}

Following~\cite[\S3.3]{MX25}, we may in fact work in a more general
setting than that of Schottky groups, and this will be convenient for our treatment of hyperelliptic
Mumford curves.
Let $\Gamma\subset \PGL_2(K)$ be a finitely generated discontinuous group. 
By Lemma~\ref{L:DiscSchott}, we can write $\Gamma = \bigcup_{i=1}^m W\gamma_i$,
where $W$ is Schottky and $\gamma_1,\ldots,\gamma_m \in \Gamma$.
\begin{defn}\label{D:thetaGamma}
  The \textit{theta function (or theta pairing) with respect to $\Gamma$} is the pairing on
  $Z^0(\Omega_\Gamma)$ defined by
  \begin{equation}\label{DEGamma}
    (D,E)_\Gamma \colonequals \prod^m_{i=1}(D,\gamma_iE)_W\,.
  \end{equation}
\end{defn}
This pairing is well-defined by~\cite[Proposition~3.6]{MX25}.

\begin{rk}\label{R:WeilRec}
  In fact we can extend the pairings~\eqref{CrossRatio} (and hence \eqref{DEW} and
  \eqref{DEGamma}) to pairs $D,E$ of divisors of degree~0 rational over $K$,
  but not necessarily pointwise rational by defining 
  \begin{equation}\label{}
    (D,E)\colonequals f(E)\,;\quad D = \div(f)\,.
  \end{equation}
  By Weil reciprocity, everything above remains valid in this more general
  setting.
\end{rk}
\subsubsection{Computing $\theta$-functions naively}\label{subsec:naive}

Definition~\ref{D:thetaW} (and~\ref{D:thetaGamma}) immediately suggests a
method to compute the theta function with respect to a (finite index
supergroup of a) Schottky group: approximate $(D,E)_W$ via
$(D,E)_{W_n}$ for $n$ large enough. We call this the~\textit{naive
approach}. It requires a study of the error term, which is discussed in
detail in~\cite[\S3.4]{MX25}, see also~\cite[Theorem~3.6]{MR15} for the
special case of the theta function $u_\gamma$ discussed below
in~\S\ref{subsec:periods}.  To get a precision of $O(p^N)$ one needs to take
an approximation with $n$ of size roughly linear in $N$. The number of elements in
$W_n$ is exponential in $n$, and hence this approach yields an exponential algorithm.

\subsubsection{The iterative approach}\label{subsec:iterative}

In~\cite{MX25}, Masdeu and Xarles introduced an iterative
approach to compute theta functions with respect to Schottky groups. This
yields a polynomial time algorithm to compute theta functions. Moreover,
the output of the algorithm is a locally analytic function that can be
either evaluated at a point, but also allows for the computation of the
derivative of the theta function. This is crucial for the algorithms that
rely on a Newton iteration, see~\S\ref{subsec:logderivs}.

More precisely, if $W$ is a Schottky group with good fundamental domain
$\mathcal{F}$, one can consider the affinoid algebra $\mathcal{O}(\mathcal{F})$ of
rigid analytic functions on the connected affinoid $\mathcal{F}$. The algorithm
in~\cite{MX25} produces a rational function $\phi$ supported on $\mathcal{F}$, and
a function $G \in \mathcal{O}(\mathcal{F})^\times$ such that the theta function
$(z-z_0,E)_{W}$ is approximated (to precision $O(p^N)$), for $z$ and $z_0$ in $\mathcal{F}$ by  $\phi(z) G(z)$.

\subsection{Mumford curves}\label{subsec:mumford}

\begin{defn}\label{D:}
	Let $C/K$ be a nice curve. In case it has a semistable model
  $\mathcal{C}/\O$ such that
	\begin{itemize}
		\item the normalization of any irreducible component of the special fiber $\mathcal{C}_k$ is isomorphic to $\mathbb{P}^1_k$, and
		\item every double point of $\mathcal{C}_k$ is $k$-rational with two
      $k$-rational branches,
	\end{itemize}
	we say that $C$ has \textit{split degenerate reduction}.
\end{defn}

\begin{thm}[Mumford,~\cite{Mu72}]\label{T:SchMum}
	Let $W\subset \PGL_2(K)$ be a Schottky group of rank $g\ge 2$. Then there
  is an isomorphism of rigid-analytic spaces
  $\Omega_W/W\simeq C_W^{\an}$, where $C_W^{\an}$ is the rigid analytification
  of a nice curve $C_W/K$ of genus $g$ with split degenerate reduction.
  Moreover, this association induces a bijection:
	\[
	\begin{aligned}
	\Bigg\{\begin{aligned}
		&\text{conjugacy classes of Schottky} \\
		& \ \ \ \ \ \text{groups in } \PGL_2(K)
	\end{aligned}\Bigg\}&\to \Bigg\{\begin{aligned}
		&\text{isomorphism classes of nice curves over} \\
		& \ \ \ K \text{ with split degenerate reduction}
	\end{aligned}\Bigg\}\\
	W &\mapsto C_W
	\end{aligned}
	\]
\end{thm}

\begin{defn}\label{D:}
	A nice curve $C/K$ is called a \textit{Mumford curve} if
  $C\simeq C_W$ for some Schottky group $W\subset \PGL_2(K)$.
\end{defn}

Thanks to Theorem~\ref{T:SchMum}, Mumford curves are precisely nice curves with
split degenerate reduction. More information can be found
in~\cite[Section~3]{vdPT26}. 
We emphasize that, for a Schottky group $W\subset \PGL_2(K)$, the rigid
isomorphism $\Omega_W/W\simeq
C_W^{\an}$ is induced by an analytic covering $u\colon \Omega_W \rightarrow C_W^{\an}$. This means that there is a finite admissible covering
$\{U_j\}$ of $C_W^{\an}$ by affinoids such that 
the covering $u$ is trivial above all $U_j$. Hence we obtain:
\begin{cor}\label{C:Mumext}
Let $L/K$ be a complete field extension. Then $u\colon \Omega_W \rightarrow
  C_W^{\an}$ induces a bijection $\Omega_W(L)/W\rightarrow C(L)$.
\end{cor}
It is in general quite difficult to make the correspondence in
Theorem~\ref{T:SchMum} explicit. If generators $w_1,\ldots,w_g$ for a
Schottky group $W$ of rank $g\ge 3$ are
given, then Gerritzen and van der Put discuss in~\cite[IV.4]{GP80} how one may compute the
canonical embedding of $C_W$ using the theta function on $\Omega_W$ defined
by
\begin{equation}\label{uw}
  u_w(z) \colonequals (z-\infty, \iota(w))_W\,;\quad w \in W\,.
\end{equation}
In this case, the canonical embedding $C_W\to \BP^{g-1}_K$ is given by
\begin{equation}\label{canemb}
  z\mapsto (\dlog u_{1}(z) : \cdots : \dlog u_{g}(z))\,,
\end{equation}
where we write $u_i\colonequals u_{w_i}$ for simplicity.
For non-hyperelliptic $C_W$, one may then use this to compute equations for
$C_W$ (see {\cite[\S3.3]{MR15}}).
The hyperelliptic case is due to van der Put~\cite{vdP78} and is
reviewed below in~\S\ref{subsec:unif}; see also~\cite{vdPT26}. 
On the other hand, it is only known how to compute the Schottky group $W$ from
an equation for $C_W$ for certain hyperelliptic curves; this is due to
Kadziela~\cite{Kad07} and is reviewed and extended below in
\S\ref{S:FixedFromRam}.

\subsection{Period lattice}\label{subsec:periods}

Manin and Drinfeld~\cite{MD73} used theta functions 
to uniformize the Jacobian $J$ of a Mumford curve $C$,
as we now recall.
The reduction of $J$ is split degenerate, 
and hence $J^{\an} \simeq T/\Lambda$ for a split analytic torus $T/K$ and a
multiplicative lattice $\Lambda$
inside $T$. These objects have the following interpretation in terms of the Schottky group
$W$. Let $W^{\ab} = W/[W,W]$, where $[W,W]$ is the commutator subgroup of
$W$. Then $T$ is the analytification of the split algebraic torus with
character group $W^{\ab}$, and we  can identify  $T(K)$
with $\Hom(W^{\ab},K^{\times})$. By properties of $(\cdot,\cdot)_W$, the pairing
\begin{equation}\label{periods}
  \langle\cdot  ,\cdot \rangle_{W^{\ab}}\colon W^{\ab}\times W^{\ab} \to K^{\times},\ \ \
(\beta_1,\beta_2)\mapsto (\iota(\beta_1),\iota(\beta_2))_W,
\end{equation}
is well-defined, symmetric and bimultiplicative. Every $\beta\in W^{\ab}$ defines a map
\[\ell_{\beta}\colon W^{\ab}\to K^{\times},\ \ \ \beta'\mapsto \langle
\beta,\beta'\rangle_{W^{\ab}},\]
and gives an element of $T(K)$. Then the lattice $\Lambda\subset T$ is given
by
\begin{equation}\label{}
  \Lambda = \{\ell_{\beta}\,:\, \beta \in W^{\ab}\}\,.
\end{equation}
We can make this explicit by fixing a choice of generators $w_1,\ldots,w_g$ of
$W$, giving a basis $b_i\colonequals w_i \bmod [W,W]$ of the
free abelian rank~$g$ group
$W^{\ab}$. By duality, this choice induces a splitting $\varphi\colon
T\simeq
(\G_m^{\an})^g$ given by the homomorphisms $b_i^*$ defined by
$b_i^*(b_j)=\delta_{ij}$. Hence we have 
\begin{equation}\label{}
  J^{\an}\simeq T/\Lambda \simeq (\G_m^{\an})^g/Q_W\,,
\end{equation}
where 
$Q_W\in
K^{g\times g}$ is the~\textit{period matrix} of  $W$ (or $J$) with respect
to $w_1,\ldots,w_g$, defined by
\begin{equation}\label{}
  Q_W = (\langle b_i, b_j\rangle_{W^{\ab}})_{i,j}\,.
\end{equation}
Moreover, the Abel--Jacobi embedding $C\hookrightarrow
J$ with respect to a fixed base point $P_0\in C(K)$  lifts to 
\begin{equation}\label{}
  j_W\colon \Omega_W\to T\,;\quad z\mapsto (\gamma\mapsto (\iota(\gamma),
  z-z_0)_W)\,,
\end{equation}
where $z_0\in \Omega_W(K)$ lifts $P_0$. If $z_0=\infty$ lifts $P_0$, then we have
\begin{equation}\label{}
  \varphi\circ j_W(z) = (u_1(z),\ldots ,u_g(z))\in (L^\times)^g\,,
\end{equation}
where $L/K$ is any extension such that $z\in \Omega_W(L)$.
To sum up, we have a commutative diagram of rigid analytic morphisms
\begin{equation}\label{D1}
\begin{tikzpicture}[>=stealth, node distance=3em]
  \matrix (m) [matrix of math nodes, row sep=3em, column sep=4em]{
    \Omega_W &  & T && (\G_m^{\an})^g\\
    \Omega_W/W & & T/\Lambda && (\G_m^{\an})^g/Q_W \\
    C^{\an} & & J^{\an} \\};
  \draw[->>] (m-1-1) edge (m-2-1);
   \draw[->]  (m-1-1) edge node[below] {$j_W$}(m-1-3);
  \draw[->>] (m-1-5) edge (m-2-5);
   \draw[->]  (m-1-3) edge node[below] {$\varphi$} node[above] {$\simeq$} (m-1-5);
   \draw[->]  (m-2-3) edge node[below] {$\varphi$} node[above] {$\simeq$} (m-2-5);
  \draw[->>] (m-1-3) edge (m-2-3);
   \draw[right hook-latex] (m-2-1) edge node[below] {$j_W$} (m-2-3);
   \draw[->] (m-2-1) edge node[left] {$\simeq$} (m-3-1);
   \draw[->] (m-2-3) edge node[left] {$\simeq$}  (m-3-3);
   \draw[->] (m-2-5) edge node[below] {$\simeq$}  (m-3-3);
   \draw[right hook-latex] (m-3-1) edge node[below] {$j$} (m-3-3);
\end{tikzpicture}
\end{equation}
See~\cite[VI.2]{GP80} and~\cite[\S2]{Wer96} for more details.

\section{Hyperelliptic Mumford curves}\label{S:unif_th}

The main result of this section is an explicit rigid analytic
uniformization of a
hyperelliptic curve $C$ with split degenerate reduction using theta
functions (see Theorem~\ref{T:H}), essentially
following~\cite[\S6]{vdPT26}. The first step is to find generators of the corresponding
Schottky group. This was done in~\cite{vdP78};  the idea is to find a
theta function that gives a suitable rigid analytic uniformization 
$F\colon \Omega\to \BP^1$ and lift it via the double cover $\pi\colon C\to \BP^1$. The resulting 
Schottky group $W$ is called a Whittaker group. A set of generators can be
described in terms of the branch points of $\pi$ and and the theta
function $F$. In this section, we describe how to construct another theta function $H$ such that
$(F,H)\colon \Omega\to C$ is a rigid analytic uniformization (see
also~\cite[\S6]{vdPT26}).
Finally, we discuss the case where the genus of $C$ is two in more detail
in~\S\ref{S:g2mum}.

We keep the notation of the previous section.

\subsection{Hyperelliptic uniformization via Whittaker groups}\label{subsec:unif}

 Suppose that $g\ge 2$ and that $\T=\{a_0,b_0,\dots, a_g,b_g\}\subset \BP^1(K)$ is in good
 position. Let $s_i\in \mathrm{PGL}_2(K)$ be the unique
  matrix of order $2$ with fixed points $\{a_i,b_i\}$. 
 Then the group $$\Gamma(\T) = \pair{s_0,\dots,s_g}
 \equalscolon \Gamma$$ is not free, hence not Schottky. But
 it is discontinuous, and 
 we can construct an index 2 subgroup $W\subset \Gamma$ that is Schottky,
 as we now explain. 

Let $\Sigma$ be the set of limit points of $\Gamma$ and let
$\Omega_\Gamma =
\BP^1\setminus\Sigma$ be the set of ordinary points of $\Gamma$. 
There is a natural morphism $\varphi\colon \Gamma\to
\Z/2\Z$, sending all $s_i$ to $-1$. Let $W$ denote its kernel and for
$i=0,\ldots,g-1$ let
\begin{equation}\label{wi}
  w_{i+1}\colonequals s_is_g\,.
\end{equation}
\begin{prop}[van der Put,~\cite{vdP78}]\label{P:W}
    The group $W$ is Schottky of rank $g$, and $W$
     is freely generated by $w_1,\ldots,w_{g}$.    
\end{prop}
We will denote $\Omega =\Omega_W$ for the rest of this
section; we have $\Omega=\Omega_\Gamma$ by Proposition~\ref{L:DiscSchott}.
\begin{defn}\label{D:}
  We call the group $W$ a {\em ($p$-adic) Whittaker group}.
\end{defn}

As explained by van der Put~\cite{vdP78}, the group $\Gamma$ provides an explicit uniformization of $\BP^1$, as we
now recall.
Fix distinct $a,b \in \Omega$ such that  
$\infty \notin \Gamma a \cup \Gamma b$.
In contrast to other treatments in the literature, we do allow that $a,b$ or $\infty$ are among the $a_0,b_0,\dots a_g,b_g$.
We will use the trivial observation that for any group homomorphism
$c\colon \Gamma \to G$, we have $\#c(\Gamma)\le 2$.

Consider the theta function
\begin{equation}\label{F}
  F(z) \colonequals F_{a,b}(z) \colonequals (a-b, z-\infty)_\Gamma =
  \prod_{\gamma\in\Gamma}\frac{z-\gamma a}{z-\gamma b}\,.
\end{equation}
\begin{lemma}\label{L:FGamma}
 The function $F$ is $\Gamma$-invariant.
\end{lemma}
\begin{proof}
  Let $\Omega^*\colonequals \Omega - \displaystyle\bigcup^g_{i=0} \big(\Gamma\cdot a_i\cup \Gamma\cdot
  b_i\big)\,.$
  The proof of~\cite[Theorem~2.5]{vdPT26} shows the result for 
$(a,b)\in \Omega^* \times \Omega^* - \Delta$. 
  By allowing the points $a,b$ to move and using connectedness, we extend
  this to all distinct $a,b\in \Omega$ such that $\infty \notin \Gamma a\cup \Gamma
  b$.
 \end{proof}
\begin{cor}\label{L:P1}~\cite[Proposition~4]{vdP78}.
The theta function $F$ induces an isomorphism $\tilde{F}\colon
  \Omega/\Gamma\simeq \BP^{1,an}$ of rigid analytic spaces.
\end{cor}

\begin{rk}\label{R:}
  For two different choices for $(a,b)$, the resulting theta functions $F$
  define the same field extension.
\end{rk}

Since $W$ is a Schottky group with set of ordinary points $\Omega$, the
rigid quotient $\Omega/W$ is a Mumford curve and we let $C\colonequals C_W$ denote the
corresponding algebraic curve. The natural
covering $\phi\colon \Omega/W \to \Omega/\Gamma$ has degree~2, so we find that $C$ is 
hyperelliptic.
The $\Gamma$-invariant
meromorphic function $F$ induces an element $x\in K(C)$. By Riemann--Roch,
there is a function $y\in K(C)$ such that $K(C)=K(x,y)$ and
$y^2$ is a polynomial in $x$ of degree $2g+1$ or $2g+2$. 
In other words, we have a commutative diagram
 \[
\begin{tikzpicture}[>=stealth, node distance=3em]
  \matrix (m) [matrix of math nodes, row sep=3em, column sep=4em]{
    \Omega & & \\
    \Omega/W & & \Omega/\Gamma \\
    C^{\an} & & \mathbb{P}^{1,\an} \\};
  \path[->]
    (m-1-1) edge (m-2-1)
    (m-1-1) edge (m-2-3)
    (m-2-1) edge node[below] {$2:1$} (m-2-3)
    (m-2-1) edge node[left] {$\simeq$} node[right] {$\tilde{u}$}(m-3-1)
    (m-2-3) edge node[left] {$\simeq$} node[right] {$\tilde{F}$} (m-3-3)
    (m-3-1) edge node[below] {$2:1$} (m-3-3);
\end{tikzpicture}
\]
 The composite  map $u\colon \Omega\to\Omega/W\simeq C^{\an}$ is an analytic covering as
 in~\S\ref{subsec:mumford}.

By Corollary~\ref{L:P1} the set $\B_C$ of branch points for the
map $(x,y)\mapsto x$ is given by
\begin{equation}\label{}
  \B_C = \{F(a_0),F(b_0),\dots ,F(b_g)\}\,.
\end{equation}
 Since the function $F$ is defined over $K$ and $K$ is complete, the branch points are also defined over $K$.   
 This leads to  the following result.
\begin{prop}\label{P:hyp_eqn}~\emph{(\cite[Theorem~5]{vdP78},~\cite[Proposition~6.1]{vdPT26})}.
There exists a constant $c\in K^\times$ such that
  $C$ has an equation 
  \begin{equation}\label{Ceqn}
  Y^2 = c\cdot \prod^g_{i=0}\big(X-F(a_i)Z\big)\big(X-F(b_i)Z\big)
  \end{equation}
  in the weighted projective plane  over $K$ with weights $1, g+1, 1$
  attached to $X,Y,Z$, respectively.
\end{prop}
 This is not surprising, since the curve $C$ is determined by the branch locus of its 2-1
 covering $\pi\colon C \to \BP^1$, and the latter is isomorphic to $\Omega/\Gamma$ via $F$.
 Conversely, every split 
 degenerate hyperelliptic curve over $K$ is parametrized by a
 Whittaker group, unique up to conjugation~\cite[Theorem~2.7]{vdPT26}.

\subsection{Constructing the function $y$ using theta functions on
$\Omega$}\label{subsec:y}

In order to give an explicit Mumford uniformization of the hyperelliptic
Mumford curve $C$,
we will now lift the function $y\in K(C) = K(x,y)$ to a
meromorphic function on $\Omega$.
This was first done by van der Put and Top in~\cite[\S6]{vdPT26}; we
follow their proof, but we work in slightly greater generality and use the language of theta
 pairings introduced in~\cite[\S3]{MX25} and~\S\ref{S:Theta}.

\begin{lemma}\label{L:ytilde}
The function $y\in K(C)$ is
induced by a unique meromorphic function $\tilde{y}$ on $\Omega$.
It satisfies the following properties and is uniquely
determined by them up to a constant:
\begin{enumerate}
  \item\label{Winv} $\tilde{y}$ is $W$-invariant.
  \item\label{notGammainv} $\tilde{y}(s_iz)=-\tilde{y}(z)$ for all $i\in\{0,\ldots,g\}$.
  \item\label{divy} The divisor of $\tilde{y}$ is the $W$-orbit of $ D
    \colonequals\displaystyle\sum^{2g}_{i=0}( a_i+ b_i)
    -(g+1)b-(g+1)s_gb\,.$
\end{enumerate}
\end{lemma}

\begin{proof}
  The existence and uniqueness of the function $\tilde{y}$ are clear. It has to be
  $W$-invariant since $C^{\an}\simeq \Omega/W$.
  The second property follows since $y\in K(C)$ is sent to $-y$ by the hyperelliptic involution.
Moreover, the divisor of $y^2\in K(\mathbb{P}^1)$ is 
  $\sum_{i=1}^{2g+2} (F(a_i)+F(b_i))-(2g+2)\infty$, and we have $F(b)
  =F(s_gb)=\infty$. 
\end{proof}
We now want to construct the function $\tilde{y}$ explicitly. As a starting point, we
take~(\ref{divy}); it is satisfied by the theta function
\begin{equation}\label{}
  z\mapsto (z-\infty, D)_W\,.
\end{equation}
However, a computation shows that this function is not $W$-invariant. To
make it $W$-invariant, we modify it
without changing its divisor. 
Define $w'\colonequals w_0\cdots w_{g-1}$ and set 
\begin{equation}\label{}
  E \colonequals [D] + \iota(w') \in Z^0(\Omega_W(K))_W \,,
\end{equation}
  where $[\cdot]$ denotes the class in $Z^0(\Omega_W(K))_W$.
We will show below that the following theta function is in fact equal to
$\tilde{y}$ up to constant:
\begin{equation}\label{}
  H(z) \colonequals (z-\infty, E)_W\,.
\end{equation}
\begin{lemma}\label{L:siEE}
  The class $s_iE - E\in Z^0(\Omega_W(K))_W$ is trivial for every
  $i\in \{0,\ldots,g\}$.
\end{lemma}
\begin{proof}
  We first rewrite $E$ in a convenient way.
  Recall from~\eqref{wi} that for $j=0,\ldots,g-1$, we have $w_{j+1} =s_{j}s_g$ and that $s_j$ fixes $a_j$ and $b_j$ for all $j$.
In particular, we have
$$
  \iota(w_{j+1}) =  [w_{j+1}b_j - b_j] = [s_js_gb_j - b_j] = [s_js_gb_j - s_jb_j] =
  [s_gb_j-b_j]\,.
$$
Since $\iota$ is a homomorphism, we get
$$
  \iota(w') = \sum^{g}_{j=0} [s_gb_j-b_j]\,,
$$
implying
  \[
    E 
    = [a_0+s_gb_0 +\ldots+a_g+s_gb_g - (g+1)b - (g+1)s_gb]\,.
  \]
  We have
  \begin{align*}
    s_i(a_j+s_gb_j - b - s_gb) &= s_is_ja_j+s_is_gb_j - s_ib - s_is_gb \\
    &= s_is_gs_gs_ja_j+s_is_gb_j -s_is_gs_gb - s_is_gb =
    w_{i+1}w_{j+1}^{-1}a_j+w_{i+1}b_j - w_{i+1}s_gb-w_{i+1}b\,,
  \end{align*}
  for all $j\in \{0,\ldots,g\}$, so that 
  $$[s_i(a_j+s_gb_j - b-s_gb)]=[w_{j+1}^{-1}a_j+b_j-s_gb-b]\,.$$
  Hence we conclude that
  \[s_iE - E = \sum^{g}_{j=0}[w_{j+1}^{-1}a_j - a_j + b_j - w_{j+1}^{-1}b_j] = 0\,. \qedhere\] 
\end{proof}
 
\begin{lemma}\label{L:H}
There  is a homomorphism 
 $\kappa\colon \Gamma \to \{\pm 1\}$ such that $H(\gamma
  z)\kappa(\gamma)=H(z)$ holds for every $\gamma \in \Gamma$. 
\end{lemma}
 
\begin{proof}
  For $\gamma \in \Gamma $, define 
  $$\kappa(\gamma)\colonequals (\gamma \infty - \infty, E)_W\,.$$
  Then 
  \begin{equation}\label{}
    H(\gamma z) = (\gamma z - \gamma\infty,E)_W\cdot(\gamma\infty - \infty, E)_W =
    (\gamma z - \gamma\infty,E)_W \cdot\kappa(\gamma)\,.    
  \end{equation}
  and $\kappa$ is a homomorphism, since for $\gamma,\delta\in \Gamma$, we have
  \[
    \kappa(\gamma\delta) =(\gamma\delta \infty -\infty, E)_W =(\gamma\delta \infty -\delta\infty, E)_W\cdot (\delta \infty -\infty, E)_W
    = \kappa(\gamma)\kappa(\delta)\,.
  \]

  For $w\in W$, we obtain $(wz - w\infty, E)_W = (z-\infty, E)_W$ by
  definition of $(\cdot,\cdot)_W$, and hence $H(wz) = H(z)\kappa(w)$.
  But if $i\in \{0,\ldots,g\}$, then we have
  \begin{equation*}
    H(s_i z) = (s_iz - s_i\infty,E)_W\cdot(s_i\infty - \infty, E)_W =
    (z-\infty, s_i E)_W\cdot\kappa(s_i) 
  \end{equation*}
  by Lemma~\ref{L:thetaconj}, since $s_i$ normalizes $W$.
   Lemma~\ref{L:siEE}    implies that 
   $$
    (z-\infty, s_iE)_W = (z-\infty, E)_W = H(z)\,,
   $$
   and therefore $H(s_iz) = H(s_i)\kappa(s_i)$.
\end{proof}
 
We combine the results of this subsection into the following uniformization
theorem.
\begin{thm}\cite[Proposition~6.2]{vdPT26}\label{T:H} \hfill
  \begin{enumerate}
    \item\label{Winv}  The function
 $H$
 is $W$-invariant. 
 \item\label{sinoninv} We have $H(sz) = -H(z)$ for all $s\in \Gamma\setminus W$.
 \item\label{HFeqn}
Let $H$ also denote  the induced element in the function field of 
      $\Omega /W$. Then there is a constant $c\in K^\times$ such that  
      \begin{equation*}
        H^2=c\prod \big(x-F(a_i)\big)\big(x-F(b_i)\big)\;\;\text{where}\;
        x-\infty\colonequals 1\,.
      \end{equation*}
    \item\label{unif}
Over $\bar{K}$ the hyperelliptic curve $C\colonequals C_W$  is uniformized by 
      \begin{equation}\label{}
  u\colon \Omega\to  C\,;\quad  z \mapsto \big(F(z),c^{-1/2}H(z)\big)\,. 
      \end{equation}
  \end{enumerate} 
\end{thm}
\begin{proof} 
See the proof of \cite[Proposition~6.2]{vdPT26}, which we briefly recall.
  Let $\tilde{y}$ be the lift of $y$ to $\Omega$, as in
  Lemma~\ref{L:ytilde}. The function $m\colon \Omega\to \C_p$ defined by
  $m(z)\colonequals\frac{H(z)}{\tilde{y}(z)}$ has no zeros or poles by
construction.
  Hence the function $z\mapsto m(z)^2$ does not, either, and it is
  $\Gamma$-invariant by Lemma~\ref{L:H}.
    Thus $m(z)^2$ is constant. 
  Since $m$ is holomorphic on $\Omega$
  and $\Omega$ is connected, $m(z)\equalscolon c\in K^\times$ itself is constant, and the
  proposition follows from Lemma~\ref{L:ytilde} and Proposition~\ref{P:hyp_eqn}. 
  \end{proof}
  \begin{rk}\label{R:}\hfill
    \begin{enumerate}
      \item 
    Theorem~\ref{T:H} was proved in~\cite[\S6]{vdPT26} under the
    assumption that $\infty$ is not in the $\Gamma$-orbit of any fixed
    point.
  \item 
    We can compute $c$ by evaluating $F$ and $H$ at a suitable point $z\in\Omega$.
  \item According to~\cite[Observation~6.3]{vdPT26}, the constant $c$ is always a square in
    $K$ when $g=2$, and for various configurations of branch and
    fixed points in genus~3; moreover, the authors expect $c$ to always be
        a square. In this case Theorem~\ref{T:H} (\ref{unif}) yields a univormization of $C$ over
        $K$. This was the case in all our examples.
    \end{enumerate}
  \end{rk}
   \begin{cor}\label{C:}
    Let $L/K$ be a complete field such that $c=s^2$ for some $s
    \in L$. Then $(F,s^{-1}H)$ induces a
    bijection
    $u\colon \Omega(L)/W\rightarrow C(L)$.
  \end{cor}
  \begin{proof}
    This follows from Theorem~\ref{T:H} and Corollary~\ref{C:Mumext}.
  \end{proof}
\subsection{Lifting points on hyperelliptic Mumford
curves}\label{S:lifting}
Given a point $P=(x,y)\in C(K)$, where $C/K$ is a hyperelliptic Mumford
curve uniformized by a Schottky group $W$, we want to compute a
\textit{lift} $z\in \Omega(K)=\Omega_W(K)$ such that $u(z)=P$.
We assume that we already know suitable generators $s_0,\ldots,s_g$ of $\Gamma$
(and hence generators $w_1,\ldots,w_g$ of $W$); their computation will be
discussed in \S\ref{S:FixedFromRam} below.
\begin{cor}\label{C:unif}
  Let $z\in \Omega(K)$ such that $F(z) = x$. If $H(z)=y$, then $u(z)=P$.
  Otherwise $u(s(z))=P$ where $s\in \Gamma$ is any matrix of order~2.
\end{cor}
\begin{proof}
  This is immediate from Theorem~\ref{T:H}.
\end{proof}
Corollary~\ref{C:unif} suggests the following method to compute $z\in
\Omega(K)$ such that $u(z)=P$:

\begin{algorithm}[H] \label{A:LiftPoints}
	\caption{\bf Rigid uniformization of a point on a hyperelliptic Mumford
  curve}
  \KwIn{An affine point $P=(x,y) \in C(K)$, where $C/K$ is a hyperelliptic
  Mumford curve uniformized via $u=(F,s^{-1}H)\colon \Omega/W\to C$ over $K$}
  \KwOut{A point $z\in \Omega_W(K)$ such that $u(z) = P$. }
\begin{enumerate}
  \item Compute $z\in \Omega(K)$ such that $F(z) = x$ using Newton
    iteration.
  \item If $y=0$, then return $z$
  \item If $y\ne 0$, compute $H(z) \pmod {\pi^{v(y)+1}}$. If this equals $y
    \pmod{\pi^{v(y)+1}}$, return $z$, else return $s_0z$.
\end{enumerate}
\end{algorithm}

In practice, this approach requires algorithms to compute the theta pairing $(\cdot, \cdot)_W$; see~\S\ref{subsec:naive}
and~\S\ref{subsec:iterative}. 
We can in principle compute $H(z)$ to any desired precision, but we only
compute enough digits to test whether $y=H(z)$ or $y=H(s_0z)$.
\subsubsection{Computing derivatives of
$\theta$-functions}\label{subsec:logderivs}
Newton iteration requires the computation of derivatives of theta
functions. This can be done by differentiating the infinite product
defining the theta function or, better yet, by using the fact that $F' = F \dlog F$.
We can compute $\dlog F$ both via the naive and the iterative methods. For example,
the iterative method yields a representation of $F$ essentially as a product of a rational
function times several power series (see \cite[\S 4.3]{MX25}) (with different uniformizing parameters), so
to compute $\dlog F$ from this representation
one differentiates the formal logarithm of each of the power
series, and then adds the results to the logarithmic derivative of the rational function.


  \subsection{Position of fixed points and branch points}\label{}
Recall that in the present section, we have assumed that the set
$\T=\{a_0,b_0,\dots, a_g,b_g\}\subset \BP^1(K)$ is in good position (see
Definition~\ref{elts:gens_in_good_pos}). We now discuss this in more
detail.

\begin{cor}\label{cor:KadzielaShape}
  Let $W$ be a Whittaker group arising from the fixed points $\T$ and 
	let $C=C_W$ be the corresponding hyperelliptic Mumford curve with branch points
  $\B=\B_C$. 
  Then
  there exist $\sigma_{\B},\sigma_\T\in \PGL_2(K)$ such that
	\begin{itemize}
		\item 
      $\sigma_{\B}(\B) = \{0,r_0,\dots,r_{2g-2},1,\infty\}$
		with $0 < |r_0| < |r_1| \leq \cdots \leq |r_{2g-2}| < 1$, and 
		\item
		$\sigma_\T(\T) = \{0,b_0,\dots,a_i,b_i,\dots,a_{g-1},b_{g-1},1,\infty\}$
		with $0 < |b_0| < |a_1| \leq \cdots \leq |b_{g-1}| < 1$.
	\end{itemize}
\end{cor}

\begin{proof}
  It suffices by Lemma~\ref{L:niceRos} to check that for $S\in\{\B,\T\}$,
  every line in $  \Red_S(\mathbb{P}^1_K)$ separates at most two points.
  When $S = \T$, this
  follows from Lemma~\ref{L:necescond}. When $S = \B$, this follows from the
  second statement in \cite[Theorem~V.3.1]{GP80} 
\end{proof}

\begin{lemma}\label{L:}
  Suppose that $T$ is in strong Kadziela position. Then
  $\BP^1(K)\setminus (\calB_0\cup \calC_0\cdots\cup \calB_g\cup \calC_g)$ is a good fundamental
  domain for $W$, where $\calC_i=s_g(\calB_i)$ is the disk with center $2-c_i$ 
  and radius $r_i$ (see~\eqref{rici}). 
\end{lemma}

 \subsubsection{Genus~2 Mumford curves}\label{S:g2mum}

Suppose that $C/K$ has  split degenerate reduction and is of genus $g=2$. There are three possible stable reductions of $C$:
\begin{enumerate}\label{stabtypes}
  \item[(a)]\label{staba} two projective lines 
   intersecting transversally in $3$ points,
 \item[(b)]\label{stabb} a genus~0 curve with two nodes, and
 \item[(c)]\label{stabc}  two genus~0 curves with one node each, intersecting transversally in a smooth point.
\end{enumerate}
They are respectively shown in Figure~\ref{fig:Types_abc}, which is modified from a figure in \cite{CMR11}:
\begin{figure}[ht]
	\centering
	
	\begin{subfigure}{0.32\textwidth}
		\centering
		\includegraphics[width=0.7\linewidth]{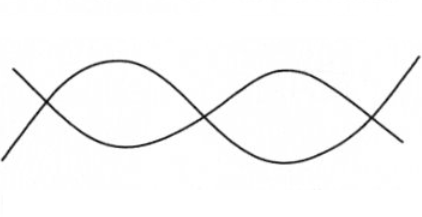}
		\caption{Type~(a)}
	\end{subfigure}
	\hfill
	\begin{subfigure}{0.32\textwidth}
		\centering
		\includegraphics[width=0.65\linewidth]{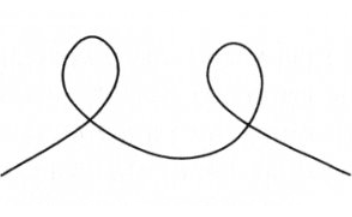}
		\caption{Type~(b)}
	\end{subfigure}
	\hfill
	\begin{subfigure}{0.32\textwidth}
		\centering
		\includegraphics[width=0.45\linewidth]{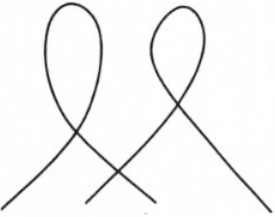}
		\caption{Type~(c)}
	\end{subfigure}
	
	\caption{Stable reductions of genus~$2$ curves with split degenerate reductions}
	\label{fig:Types_abc}
\end{figure}

We denote these types by (a), (b) and (c) as in~\S\ref{S:posg2}, since the
position of the set of branch points is of the respective type,
see~\cite[Chapter~IX, (2.5.3)]{GP80}.
Hence $\B_C$ is in closed disk position if and only if the
stable reduction is of type~(a).
In this case $\B_C$ is in good position, and we can move it to strong Kadziela
position.

\begin{ex}\label{E:KadBad}
The following curve has branch points in Kadziela position that are not in strong Kadziela position:
$$ X_1\colon y^2 = x(x-5^3)(x-5)(x-195)(x-1).$$
  Its stable reduction is of type~(b) as one sees by looking at the
  branch points $0, 5^2, 1, 39, 1/5, \infty$ of the curve obtained by $x\mapsto
  x/5$.
  One can easily show that there is no transformation of $X_1$ with the property that $\B_{X_1}$ is
  in weak Kadziela position. On the other hand, the set of fixed points is
  $\T = \{0,b_0,a_1,b_1,1,\infty\}$, where
  \[b_0 = 5^{3} \cdot 87495069218 + O(5^{20}),\ \ \ a_1 = {5 \cdot 7806971503561 + O(5^{20})},\ \ \ b_1 = {5 \cdot 12203741012063 + O(5^{20})};\]
  see Example~\ref{Ex:KadExamples}, and hence $d_{01} = 1/25$ and $d_{10} =
  1$, that is, the set $\T$ is in weak Kadziela position, and is of type
  (b). 
  
  This shows that the position of the branch and fixed points need not be the same and thus gives a
  counterexample to to the conjecture in \cite[\S IX,
  page~282]{GP80}, repeated in~\cite[Conjecture~3.1]{Kad07}. The same
  situation occurs for the curves $X_2,\ldots,X_6$
  in~\cite[\S7.4]{Kad07}.
\end{ex}

\begin{rk}\label{R:}
  It was already shown by Yelton in~\cite{Yel23} that the above-mentioned conjecture does
  not hold in general, but that a modifed (and stronger) version does hold.
  Namely, one has to assume that $\T$ is \textit{optimal} in the sense
  of~\cite{Yel23}. In fact Yelton works in the more general setting of
  superelliptic Mumford curves.
\end{rk}

\section{Computing the Schottky group of a hyperelliptic Mumford
curve}\label{S:FixedFromRam}

No general algorithm to compute the Schottky group corresponding
to a curve with split degenerate reduction is known. If the curve
is hyperelliptic and the fixed points are in strong Kadziela position
then such an algorithm was provided by Kadziela in his thesis~\cite{Kad07}, and we
recall it here. We also extend it to the case where the fixed points are in
weak Kadziela position. Finally, we describe a new method based on
multivariate Hensel lifting that is much faster in practice,
see~\S\ref{S:Hensel}. While we cannot prove that our new method always
gives the correct result, we show how to check when it terminates, and when the
fixed points are in weak Kadziela position, then the correctness of the
result can be determined. 

Let $C/K$ be a hyperelliptic Mumford curve uniformized via a Whittaker
group $W$.
Following Kadziela~\cite{Kad07}, we assume in this section that both the
set of fixed points $\T$ of the generators $s_0,\ldots,s_g$ of the
group $\Gamma$ and the set $\B_C$ of branch points of the curve $C$ are in
Kadziela position, so in particular, $0,1,\infty\in \T\cap \B_C$.
We also assume that $\T$ is in good position.
The key idea is that, by Proposition~\ref{P:hyp_eqn}, the theta function $F(z) = F_{0,1}(z)= \theta_{\Gamma}(0,1\;z)$
maps $\T$ bijectively onto $\B_C$, and we know the
latter. Since we can find the $s_i$ from their fixed points, we only need
to compute the preimages of $\B_C$ under $F$. The
problem is that even the definition of $F$ depends on $\T$, which is
precisely what we are trying to compute!
However, once we have computed $\T$ to a certain precision, we may
then approximate the function $F$ to a certain precision.

\subsection{Kadziela's algorithm}\label{S:Kadziela}

Kadziela's approach is as follows
\begin{enumerate}
  \item\label{Step:Initial} Compute initial approximations of $\T$.
  \item\label{Step:Lifting} Lift $\T$ digit by digit, by trying all possible next digits,
    evaluating a suitable  approximation $L_n$ of $F$ in the candidates and
    checking which choice of next digits for each $z\in \T$ minimizes
    $L_n(z)-r$ where $z$ runs through $\T$ and $r$ runs through $\B_C$.
\end{enumerate}

His algorithm is reproduced with some minor errors in~\cite[Algorithm~5.5]{MR15}. 
Kadziela showed~\cite[Theorem~6.10]{Kad07} that it is guaranteed
to give us $\T$ to any
desired precision after finitely many steps, provided that $\T$ is indeed in
strong Kadziela position.
We extend this result to weak Kadziela in Theorem~\ref{T:approx}.
Kadziela used his algorithm to compute $\T$ in some examples
in~\cite[\S7.4]{Kad07}. We found that in these examples, $\T$
is in weak Kadziela position, but not in strong Kadziela position.

We have implemented a version of Kadziela's algorithm, using both the naive
and the iterative approach to computing $F$, and used it to compute
$\T$ in a number of examples.  
The algorithm requires to try up to $(\#k)^{2g-1}$ next digits, and hence
is quite slow if $\#k$ or $g$ is large.
The computation for each possible next digit is independent
of one another, so the algorithm can be parallelized (and we have done so
in our implementation of Kadziela's algorithm).
Note however that for fixed $\#k$ and $g$ this algorithm is already polynomial time in the
number of digits of precision when using the iterative approach to computing $F$, even if in
practice it is too slow to be useful for large $\#k$ or $g$.

Our new method discussed in~\S\ref{S:Hensel} gives a different approach to
Step~\eqref{Step:Lifting} based on Hensel lifting.
In both algorithms, we do not strictly need to know an initial approximation; we can
simply try all possible approximations until we find one that lifts
correctly. In practice, this can be very slow, so that it is 
preferable to use an initial approximation, if one is available. Studying these first
approximations is the goal of the next subsection.

\subsection{Approximations}\label{S:Initial}
Kadziela solved Step~\eqref{Step:Initial}
under the fairly restrictive condition that $\T$ is in strong Kadziela position.
Recall from~\S\ref{S:posg2} (see also~\cite{vdPT26} for details)   that in genus~2 this is only satisfied by one of the three possible
types of position. We weaken his conditions to weak Kadziela position.
For $\gamma\in \Gamma$, we write $\ell(\gamma)$ for the length of $\gamma$
as a reduced word in the generators $s_0,\ldots,s_{g}$.

\begin{thm}\label{T:approx}
Suppose that $\T$ is in weak Kadziela position. 
Then $F \colonequals F_{0,1}$  satisfies 
		$F(0) = 0$, $F(1) = \infty$, and $F(\infty) = 1$.
 Moreover, for $n\ge 0$, let 
$$L_n(z) = \prod_{\ell(\gamma) = n} \frac{z - \gamma(0)}{z -
  \gamma(1)}\,.$$
Then we have 
  $F(z) = \prod_{n = 0}^{\infty} L_n(z)$.
	For $z\in \T\setminus\{0,1,\infty\}$, we have	
	\begin{itemize}
		\item $F(z)\equiv 0\mod\pi$,
		\item $F(z) \equiv \begin{cases}
			-4b_0 \displaystyle\prod_{i=1}^{g-1}\prnths{1 - \left(\frac{a_i-b_i}{a_i+b_i}\right)^2}\mod\pi^2 & \text{if }\ z=b_0,  \\
			-2z \displaystyle\prod_{i=1}^{g-1}\prnths{1 + \frac{(a_i-b_i)^2}{(a_i+b_i)(2z - a_i - b_i)}} \mod\pi^2 & \text{otherwise},
		\end{cases}$
		\item $F(z) \mod\pi^t = \displaystyle\prod_{i=0}^{t} L_i(z)\mod\pi^t = \prod_{i=0}^{t} L_i(z\mod\pi^t)$ for $t\geq 3$.
	\end{itemize}
\end{thm}

\begin{rk}\label{R:}
In addition to the initial approximations, the result also tells us how
  well the finite subproducts $L_n$, evaluated in approximations of the
  fixed points, approximate the function $F$, evaluated in the actual fixed
  points. 
  This is crucial in Kadziela's solution to Step~\eqref{Step:Lifting}.
\end{rk}

We prove Theorem~\ref{T:approx} by following Kadziela's proof
of~\cite[Theorem~6.10]{Kad07}.
Fortunately, all statements in \cite[Chapter~6]{Kad07}, except
Corollary~6.9 and Theorem~6.10, are still valid under our weaker condition.

\begin{lemma}\label{lem:ri_ci}
	For $0 < i < g$, we have $|r_i| \leq |c_i|$.
\end{lemma}
\begin{proof}
	See the proof of \cite[Corollary~5.4]{Kad07}.
\end{proof}

Let  $\gamma = \gamma_{a_1}\cdots \gamma_{a_n}$,
$\gamma_{a_i}\in\{s_0,s_1,\dots,s_{g}\}$, be a reduced word of length $n\geq 2$. Let $\ell_{\gamma} = \frac{\gamma(1) - \gamma(0)}{z - \gamma(1)}$. Then $\frac{z - \gamma(0)}{z - \gamma(1)} = 1 + \ell_{\gamma}$.

\begin{prop}\label{prop:valoflgamma}
	For $z\in \T\setminus\{0,1,\infty\}$, we have $v(\ell_{\gamma})\geq n-1$.
\end{prop}
\begin{proof}
	The first claim in \cite[Corollary~6.9]{Kad07} is still true and says
	\[|\ell_{\gamma}| \leq d_{a_n0}\prod_{i = 1}^{n-1} d_{a_ia_{i+1}}d_{a_{i+1}a_i}.\]
	The result follows from this inequality and the weak Kadziela condition.
\end{proof}

\begin{proof}[Proof of Theorem~\ref{T:approx}]
	Thanks to Proposition~\ref{prop:valoflgamma}, for $n\geq 2$, $L_n(z) = 1 + \mu_n$
        with $v(\mu_n)\geq n-1$. Hence only $L_0$, $L_1$ and $L_2$ can
        contribute to $F(z) \mod\pi^2$. Since
	\[\L_0(z) = \frac{z}{z-1} = -z + \text{higher order terms},\]
	we have $L_0(z) \equiv -z \mod\pi^2$ and $L_0(z) \equiv 0 \mod\pi$. Therefore, it suffices to compute $L_1$ and $L_2$ modulo $\pi$. Using Proposition~\ref{prop:valoflgamma} again, $L_2(z) \equiv 1 \mod\pi$. Regarding $L_1$, we have
	\[L_1(z) = \prnths{\frac{b_0 - 2}{b_0 - 2 + \frac{b_0}{z}}\frac{z-2}{z-1}}\prod_{i = 1}^{g - 1}\prnths{1 + \frac{r_i^2}{c_i(1 - c_i)(z - s_i(1))}}\]
	as explained on \cite[page~45]{Kad07}. The first part is congruent to $4$ (resp. $2$) if $z = b_0$ (resp. $z\neq b_0$) modulo $\pi$. For the second part, take $1 \leq i \leq g-1$. \cite[Corollary~6.7]{Kad07} says that $|z - s_i(1)|\geq 1$; combining this with Lemma~\ref{lem:ri_ci}, we obtain that
	\[\left|\frac{r_i^2}{c_i(1 - c_i)(z - s_i(1))}\right|\leq 1.\]
	Moreover, one can easily check that
	\[
	\frac{r_i^2}{c_i(1 - c_i)(z - s_i(1))} = \frac{\frac{1}{2}(a_i - b_i)^2}{(a_i+b_i)\prnths{z - \frac{1}{2}(a_i+b_i) + a_ib_i - \frac{1}{2}z(a_i+b_i)}} \equiv \frac{\frac{1}{2}(a_i - b_i)^2}{(a_i+b_i)\prnths{z - \frac{1}{2}(a_i+b_i)}}\mod\pi.
	\]
	Our claim follows.
\end{proof}
As a special case, we obtain:
\begin{cor}\emph{(\cite[Theorem~6.10]{Kad07})} \label{MainApproxThm}
	Suppose that $\T$ is in strong Kadziela position.
        Then
	\begin{itemize}
		\item $F(0) = 0$, $F(1) = \infty$, and $F(\infty) = 1$.
	\end{itemize}
	For $z\in \T\setminus\{0,1,\infty\}$, we have	
	\begin{itemize}
		\item $F(z)\equiv 0\mod\pi$,
		\item $F(z) \equiv \begin{cases}
		-4b_0 \mod\pi^2 & \text{if }\ z=b_0,  \\
		-2z\mod\pi^2 & \text{otherwise},
		\end{cases}$
		\item $F(z) \mod\pi^t = \displaystyle\prod_{i=0}^{t-2} L_i(z)\mod\pi^t = \prod_{i=0}^{t-2} L_i(z\mod\pi^t)$ for $t\geq 3$.
	\end{itemize}
\end{cor}

	In fact, the proof of Theorem~\ref{T:approx} says more than what is stated in the statement itself. It allows us to guess the valuations of the fixed points in terms of the valuations of the roots. Let's make this more precise. We continue to use the notation in Theorem~\ref{T:approx} and its proof.

\begin{prop}\label{prop:v(F(z))>=v(z)}
	For a fixed point $z\in \T\setminus\{0,1,\infty\}$, we always have $v(F(z)) \geq v(z)$.
\end{prop}

\begin{proof} 
We have
\begin{align*}
	L_0(z) &= -z + \text{higher order terms} \equiv -z \mod\pi^{v(z)+1}, \\
	L_1(z) &= \prnths{\frac{b_0 - 2}{b_0 - 2 + \frac{b_0}{z}}\frac{z-2}{z-1}}\prod_{i = 1}^{g-1}\Big(1 + R_i\cdot S_i(z)\Big), \\
	L_n(z) &\equiv 1 \mod\pi \text{ for } n\geq 2,
\end{align*}
where
\begin{equation}\label{eq:Ri_and_Si}
	R_i = \frac{a_i-b_i}{a_i+b_i} \ \ \text{and} \ \ S_i(z) = \frac{a_i - b_i}{2z - (a_i+b_i) + 2a_ib_i - z(a_i+b_i)}.
\end{equation}
Then
\begin{equation}\label{eq:val_of_Fz}
v(F(z)) = v(L_0(z)) + \sum_{i=1}^{g-1} v\big(1 + R_i\cdot S_i(z)\big)
\end{equation}
from which we see that we should always have $v(F(z)) \geq v(z)$
since $v(L_0(z)) = v(z)$ and $v(R_i),v(S_i(z))$ are both nonnegative for each $i$.
\end{proof}

The valuation of $F(z)$ might be higher than the valuation of $z$. By Equation~\eqref{eq:val_of_Fz}, $v(F(z)) > v(z)$ if and only if $1 + R_i\cdot S_i(z) \equiv 0 \mod\pi$ for some index $i$. Moreover, each such index will further widen the gap between $v(F(z))$ and $v(z)$. Let us now study this contribution more closely. 

Fix an index $i\in \{1,\dots,g-1\}$. Checking if $R_i \equiv \pm 1 \mod\pi$ will be important, hence we record the following observation for later use, which is easy to prove:
\begin{lemma}\label{lem:Rimodpi}
	We have
	\begin{itemize}
		\item $R_i \not\equiv 1 \mod\pi$, and
		\item $R_i \equiv -1 \mod\pi \Longleftrightarrow |a_i| < |b_i|$. 
	\end{itemize}
\end{lemma}

\begin{lemma} \label{lem:L1(z)_contribution}
	Let $z\in\{b_0,a_i,b_i\}$. The term $1 + R_i\cdot S_i(z)$ increases the valuation of $F(z)$ if and only if $z\in\{b_0,a_i\}$ and $|a_i| < |b_i|$.	
\end{lemma}

\begin{proof} One can easily check that
\begin{equation}\label{eq:Li_second_part}
	1 + R_i\cdot S_i(z) 
	\equiv 
	\begin{cases}
		(1 - R_i)(1 + R_i)\mod\pi & \text{if } z = b_0,  \\
		(1 + R_i) \mod\pi & \text{if } z = a_i, \\
		(1 - R_i) \mod\pi & \text{if } z = b_i.
	\end{cases}
\end{equation}
The result follows from Lemma~\ref{lem:Rimodpi}.
\end{proof}

\begin{rk}
The remaining case is more complicated. Let $j$ be an index in $\{1,\dots,g-1\}$ different from $i$. We are unfortunately unable at the moment to determine $1 + R_i\cdot S_i(a_j)$ and $1 + R_i\cdot S_i(b_j)$ modulo $\pi$; the main reason is that there is no strong interaction between the fixed points with indices $i$ and those with indices $j$. 
\end{rk}

 The situation can be made even more precise and completely explicit when the genus $g$ is equal to $2$:

\begin{prop}\label{P:approxg2}
Suppose that $g=2$ and hence $C$ is given by $y^2 = x(x-1)(x-r_0)(x-r_1)(x-r_2)$ where
\[0 < |r_0| < |r_1| \leq |r_2| < 1.\]
\begin{enumerate}
	\item If $|r_1| = |r_2|$, then 
	\[v(b_0) = v(r_0) \ \ \ \text{and} \ \ \ v(a_1) = v(b_1) = v(r_1).\] 
	We have $F(b_0) = r_0$, and we can assume without loss of generality that $F(a_1) = r_1$, $F(b_1) = r_2$. Morever, the first nonzero terms of $F(b_0),F(a_1),F(b_1)$ are as follows:
	\begin{equation}
		F(z) \equiv 
		\begin{cases}
			-4b_0\left(1 - \left(\frac{a_1-b_1}{a_1+b_1}\right)^2\right) \mod\pi^{v(r_0) + 1} & \text{if } z = b_0,  \\
			-2a_1\left(1 + \frac{a_1-b_1}{a_1+b_1}\right) \mod\pi^{v(r_1) + 1} & \text{if } z = a_1, \\
			-2b_1\left(1 - \frac{a_1-b_1}{a_1+b_1}\right) \mod\pi^{v(r_1) + 1} & \text{if } z = b_1.
		\end{cases}
	\end{equation}
	\item If $|r_1| < |r_2|$, then 
	\[v(b_0) = v(r_0) - \frac{v(r_1) - v(r_2)}{2},\ \ \ v(a_1) = \frac{v(r_1) + v(r_2)}{2} \ \ \ \text{and} \ \ \ v(b_1) = v(r_2).\]
	We have $F(b_0) = r_0$, $F(a_1) = r_1$ and $F(b_1) = r_2$. Morever, the first nonzero terms of $F(b_0),F(a_1),F(b_1)$ are as follows:
	\begin{equation}
		F(z) \equiv 
		\begin{cases}
			-16b_0\frac{a_1}{b_1} \mod\pi^{v(r_0) + 1} & \text{if } z = b_0,  \\
			-4\frac{a_1^2}{b_1} \mod\pi^{v(r_1) + 1} & \text{if } z = a_1, \\
			-4b_1 \mod\pi^{v(r_2) + 1} & \text{if } z = b_1.
		\end{cases}
	\end{equation}
\end{enumerate}
\end{prop}

\begin{proof} Take a fixed point $z\in\{b_0,a_1,b_1\}$, and write
\[L_1(z) = \prnths{\frac{b_0 - 2}{b_0 - 2 + \frac{b_0}{z}}\frac{z-2}{z-1}}\ \big(1 +  R_1\cdot S_1(z)\big)\]
where $R_1$ and $S_1(z)$ are in Equation~\eqref{eq:Ri_and_Si}.
The first part is congruent to $4$ (resp. $2$) if $z = b_0$ (resp. $z = a_1, b_1$) modulo $\pi$. On the other hand, as in Equation~\eqref{eq:Li_second_part}, we have
\begin{equation}\label{eq:L1_second_part}
1 + R_1\cdot S_1(z) 
\equiv 
\begin{cases}
	 (1 - R_1)(1 + R_1)\mod\pi & \text{if } z = b_0,  \\
	 (1 + R_1) \mod\pi & \text{if } z = a_1, \\
	 (1 - R_1) \mod\pi & \text{if } z = b_1.
\end{cases}
\end{equation}

Therefore, we always have $|F(b_0)| < |F(a_1)| \leq |F(b_1)|$. In particular, $F(b_0) = r_0$ and $\{F(a_1),F(b_1)\} = \{r_1,r_2\}$.

\begin{enumerate}
	
	\item Assume $|r_1| = |r_2|$. This implies $|a_1| = |b_1|$, otherwise the absolute values of $F(a_1)$ and $F(b_1)$ would be different. Then we may assume, without loss of generality, that $F(a_1) = r_1$ and $F(b_1) = r_2$. Moreover, $1 + R_1\cdot S_1(z)$ can not be zero modulo $\pi$ for $z\in\{b_0,a_1,b_1\}$. As a result, the first nonzero terms of $F(b_0),F(a_1),F(b_1)$ are given by
	\[F(z) \equiv 
	\begin{cases}
		-4b_0(1 - R_1^2) \mod\pi^{v(r_0) + 1} & \text{if } z = b_0,  \\
		-2a_1(1 + R_1) \mod\pi^{v(r_1) + 1} & \text{if } z = a_1, \\
		-2b_1(1 - R_1) \mod\pi^{v(r_1) + 1} & \text{if } z = b_1.
	\end{cases}\]
	From this, we see that $v(b_0) = v(r_0)$ and $v(a_1) = v(b_1) = v(r_1)$.
	
	\item Assume $|r_1| < |r_2|$. This implies $|a_1| < |b_1|$, otherwise the
    absolute values of $F(a_1)$ and $F(b_1)$ would be the same. We then
    have $F(a_1) = r_1$ and $F(b_1) = r_2$. Moreover, $1 + R_1\cdot S_1(z)$ becomes zero modulo $\pi$
    for $z\in\{b_0,a_1\}$. We now determine its first nonzero term. Writing
    $v\colonequals v(a_1) - v(b_1) > 0$, we see that
	\begin{align*}
		R_1 &= \frac{\frac{a_1}{b_1} - 1}{\frac{a_1}{b_1} + 1} 
		=\left(\frac{a_1}{b_1} - 1\right)\left(1 - \frac{a_1}{b_1} + \frac{a_1^2}{b_1^2} - \cdots\right)
		= -1 + 2\frac{a_1}{b_1} - 2\frac{a_1^2}{b_1^2} + \cdots \equiv -1 + 2\frac{a_1}{b_1} \mod\pi^{v + 1}, \\
		S_1(b_0)
		&= -\frac{a_1 - b_1}{a_1+b_1 - 2a_1b_1 + b_0(-2 + (a_1+b_1))} \\
		&= -\frac{a_1 - b_1}{(a_1+b_1)\left(1 - \frac{2a_1b_1 + b_0(2 - (a_1+b_1))}{a_1+b_1}\right)} \equiv -R_1 \equiv 1 - 2\frac{a_1}{b_1} \mod\pi^{v + 1}, \\
		S_1(a_1) &= \frac{a_1 - b_1}{a_1 - b_1 + 2a_1b_1 - a_1(a_1+b_1)}
		= \frac{1}{1 - a_1} \equiv 1 \mod\pi^{v + 1}.
	\end{align*}
	Therefore, the first nonzero term of $1 + R_1\cdot S_1(z)$ for $z = b_0$ (resp. $z = a_1$) is $4\frac{a_1}{b_1}$ (resp. $2\frac{a_1}{b_1}$). As a result, the first nonzero terms of $F(b_0),F(a_1),F(b_1)$ are given by
	\[F(z) \equiv 
	\begin{cases}
		-16b_0\frac{a_1}{b_1} \mod\pi^{v(r_0) + 1} & \text{if } z = b_0,  \\
		-4\frac{a_1^2}{b_1} \mod\pi^{v(r_1) + 1} & \text{if } z = a_1, \\
		-4b_1 \mod\pi^{v(r_2) + 1} & \text{if } z = b_1.
	\end{cases}\]
	From this, we easily see that $v(b_0) = v(r_0) - \frac{v(r_1) - v(r_2)}{2}, v(a_1) = \frac{v(r_1) + v(r_2)}{2}$ and $v(b_1) = v(r_2)$. \qedhere
\end{enumerate}
	
\end{proof}

\begin{rk}\label{R:}
  Unfortunately, the results in this subsection 
  are conditional on the position of $\T$. In practice, we only have the
  branch set $\B_C$ a priori available, because we use these results
  \textit{to compute $\T$ from $\B_C$}.
  In general we cannot detect whether the approximation
  results in~\S\ref{S:Initial} apply from looking only at the branch points; see
  Example~\ref{E:KadBad}.
  However, we can determine the position of $\T$ from a crude approximation
  of $\T$, and this suffices in practice.
\end{rk}

\begin{rk}\label{R:vdPTapprox}
The relation between the fixed points and the branch points is discussed in
great detail in~\cite{vdPT26}. In particular, they prove approximation
  results that should make it possible to compute fixed points from branch
  points in greater generality, but using a different normalization from
  ours.
\end{rk}
\subsection{Hensel lifting}\label{S:Hensel}
Recall that the goal is to compute the fixed points of the generators of
the group $\Gamma$ as preimages of the branch points under the theta
function $F=F_{0,1}$. This function itself depends on the fixed points,
which we rename for convenience as
\begin{equation}\label{newT}
  \T = \{a_0=0,b_0,a_1,\ldots,a_{g-1},b_{g-1},a_g=1,b_g=\infty\}\equalscolon
\{0,t_1,\ldots,t_{2g-1},1, \infty\}\,.
\end{equation}
Recall our generators $s_0,\ldots,s_g$ of $\Gamma$, where 
\begin{equation}\label{si}
  s_0 = 
\begin{pmatrix}
t_1 & 0 \\
2 & -t_1
\end{pmatrix}, \ \ \ 
  s_i = 
\begin{pmatrix}
  t_{2i}+t_{2i+1} & -2t_{2i}t_{2i+1} \\
  2 & -(t_{2i}+t_{2i+1})
\end{pmatrix}, \	 i=1,\dots,g-1; \ \ \ 
  s_{g} = 
\begin{pmatrix}
1 & -2 \\
0 & -1
\end{pmatrix}\,;
\end{equation}
we also set $s_{g+1}=I_2$\,.
We now explain how to compute the $t_i$ from the branch points by Hensel lifting. 
We treat $\vec{t} =  (t_1,\ldots,t_{2g-1})$ as a vector of variables; then
for $i\in \{0,\ldots,g+1\}$ the matrix $s_i(\vec{t})$ is given
by~\eqref{si} and~\eqref{newT}; we also write $\Gamma(\vec{t})\colonequals
\langle s_0(\vec{t}),\ldots,s_g(\vec{t})\rangle$. We assume that we have a
first-order approximation $\vec{t}^{(0)}$ to $\vec{t}$, satisfying
\begin{itemize}
  \item $t_j-t^{(0)}_j\in \m$,
  \item $F(t^{(0)})-r_j \in \m$
\end{itemize}
for $j\in \{1,\ldots,2g-1\}$. This holds, for instance, for the
approximations discussed in the previous subsection.

Then we consider the following function
\begin{equation}\label{}
  G=(G_1,\ldots,G_{2g-1})\colon \m^{2g-1}\to \m^{2g-1}\,;\quad 
  G_j(\vec{v})\colonequals
  \bar{G}(v_j,\vec{v})-r_j\,,
\end{equation}
where 
\begin{equation}\label{}
  \bar{G}(z,\vec v) 
\colonequals
  \prod_{\gamma \in \Gamma(\vec{t}^{(0)} + \vec{v})}\frac{z-\gamma0}{z-\gamma1}\,.
\end{equation}
We then have to approximate a root $\vec{v}$ of $G$. The domain and
codomain of $G$ are both compact.

In order to use a suitable version of Hensel's Lemma, we would need to compute the
Jacobian matrix of $G$. Since the dependency of $G$ on $\vec v$ is quite complicated we
have no way of doing this in practice. Instead, we will use a very coarse approximation
to the Jacobian matrix, by only taking into account words of length zero and one in the
product defining $\bar G$.

The modified Hensel lifting algorithm uses the recursion
\[
\vec v_{n+1} = \vec v_n - T^{-1} G(\vec v_n),
\]
where $T$ is the approximated Jacobian matrix of $G$ at $\vec v_0=0$. Under suitable conditions (see~\ref{thm:hensel})
this recursion converges to a root of $G$ in linear time (instead of the usual quadratic convergence 
of Hensel's Lemma).

The following lemma, whose proof we leave to the reader, is a version of Taylor's theorem
for power series with coefficients in a non-archimedean local ring.

\begin{lemma}\label{L:Taylor1}
  Let $R$ be a local ring with a non-archimedean absolute value $|\cdot|$. If $f\colon R^n\to R^n$ is a function given by power series in $R[[x_1,\ldots,x_n]]$, then
  \[
  f(x+h) = f(x) + J_f(x)h + R(x,h),
  \]
  with $|R(x,h)| \leq |h|^2$.
\end{lemma}

We have not been able to find the following result in the literature, so we include it and a proof for completeness.

\begin{thm}
  \label{thm:hensel}
  Let $R$ be a local ring with a non-archimedean absolute value $|\cdot|$, and suppose $f\colon R^n\to R^n$ is a
  function defined by power series with coefficients in $R$ which converge on $R^n$.
  Fix $x_0 \in R^n$ and $T \in \GL_n(R)$ a matrix, and suppose that for
  some real number $0\le \lambda <1$ have
  \begin{enumerate}
    \item $|f(x_0)| \leq \lambda$.
    \item $|T - J_f(x_0)| \leq \lambda$, where $J_f$ is the Jacobian matrix of $f$,
  \end{enumerate}
  Consider the sequence $(x_n)$ defined by
  \[
   x_{n+1} = x_n - T^{-1} f(x_n).
  \]
  Then we have, for all $n\geq 0$:
  \begin{enumerate}
    \item $|f(x_n)| \leq \lambda^{n+1}$,\text{ and}
    \item $|x_n - x_0| < 1$.
  \end{enumerate}
  In particular, the sequence $(x_n)$ converges to a root of $f$.
\end{thm}

\begin{proof}
We will prove by induction on $n$ that
\begin{enumerate}
    \item $|f(x_n)| \leq \lambda^{n+1}$,
    \item $|x_n - x_0| < 1$, and
    \item $|T - J_f(x_n)| \leq \lambda$.
\end{enumerate}
The base case $n=0$ is trivial. Suppose that the result holds for some $n\geq 0$. Write $h_n =
  -T^{-1}f(x_n)$. By Lemma~\ref{L:Taylor1} we have
\begin{align*}
  f(x_{n+1})  &= f(x_n - T^{-1}f(x_n)) = f(x_n) - J_f(x_n)T^{-1}f(x_n) + R(x_n,h_n)\\
  &=(T-J_f(x_n))T^{-1}f(x_n) + R(x_n, h_n).
\end{align*}
Taking absolute values we get
\[
|f(x_{n+1})| \leq \max \{|T-J_f(x_n)| |f(x_n)|, |R(x_n, h_n)|\}.
\]
Note that $|T-J_f(x_n)| \leq \lambda$ by the third induction hypothesis, and $|R(x_n,h_n)| < |h_n|^2 \leq \lambda^{2n+2}$.
Therefore by our induction hypothesis $|f(x_{n+1})| \leq \lambda^{n+2}$, as desired.

Note also that
\[
|x_{n+1} - x_0| \leq \max\{|x_n - x_0|, |x_{n+1}-x_n|\}.
\]
By the induction hypothesis, $|x_n - x_0| < 1$. Moreover, we have $|x_{n+1} - x_n| = |T^{-1}f(x_n)| = |f(x_n)| \leq \lambda^{n+1} < 1$,
and thus $|x_{n+1} - x_0| < 1$ as well.

Finally, to check the third induction hypothesis, we write
\[
|T - J_f(x_{n+1})| \leq \max\{|T - J_f(x_n)|, |J_f(x_n) - J_f(x_{n+1})|\}.
\]
By the induction hypothesis, $|T - J_f(x_n)| \leq \lambda$. Since $J_f$
is given by power series converging on $R^n$, we have $|J_f(x_n) - J_f(x_{n+1})| \leq |x_n - x_{n+1}| < \lambda^{n+1}$ as
observed before, 
and thus $|T - J_f(x_{n+1})| \leq \lambda$ as well.
\end{proof}

In practice we have no good way of testing the hypotheses of the above result for the function $G$.
So what we do is to apply
the Hensel lifting iteration and see if it converges in at most a number of steps equal to the working precision. We have never
observed a case where this procedure fails, and in any case a posteriori we can check that the image of the fixed points coincides
with the branch points up to the desired precision, independently of any assumptions on the Jacobian. After a successful run
we will have obtained a Schottky group for which the theta function $F=F_{0,1}$ maps the fixed points to the branch points.
Whether the group is actually the uniformizing group for the curve can be
checked using the results of Section~\ref{S:Initial}, whenever the fixed
points are in weak Kadziela position. It would be interesting to extend
this using the approximation results of~\cite{vdPT26}; see Remark~\ref{R:vdPTapprox}.

\begin{rk}\label{R:}
We assumed that the branch points and fixed points both include $0,1,\infty$ for convenience and
  since this is what we assume in our algorithm. If
we do not make this assumption, the only difference is that we have to
consider $2g+2$ fixed points and $2g+2$ branch points and that we
  have to work with a more general theta function $F$.
\end{rk}


\section{Abelian logarithms and integration}\label{S:abelian}

The ($p$-adic) Abelian integral serves as a $p$-adic analogue of the
classical (real-valued) line integral. It has  numerous applications in
arithmetic and Diophantine geometry such as determining rational points on
curves~\cite{Col85Chab,MP12} and obtaining uniform bounds for their number~\cite{Sto19,KRBZ16}

Its existence is quite surprising from the point of view of classical
analysis. Namely, one can integrate differential forms locally in the
$p$-adic world, but the fact that the $p$-adic topology is totally
disconnected makes the naive analytic continuation impossible. In fact,
abelian integration on a curve is defined via the $p$-adic Lie theory of
its Jacobian. Here is a brief review of the construction, due to
Zarhin~\cite{Zar96}.

\subsection{Abelian logarithm}

Let $A/K$ be an abelian variety.
We fix an embedding $K\hookrightarrow \C_p$. Recall that every regular $1$-form on $A$ is translation-invariant; that is, $\Omega^1(A) = \Omega^1_{\inv}(A)$.
This allows us to identify naturally $\Omega^1(A)$ with the dual of
$\Lie(A)$, the Lie algebra of $A$. 
The abelian logarithm is the logarithm of the $p$-adic Lie group $A(\C_p)$,
see for instance~\cite{Zar96}.  
\begin{prop}\label{P:Zarhin}
  There is a unique homomorphism of $\Cp$-Lie groups 
  \[\log_A\colonequals \log_{A(\Cp)}\colon  A(\Cp)\to\Lie(A_{\C_p})\cong
  \Omega^1(A_{\C_p})^\vee\] 
whose linearization
  \[d\log_{A} \colon \Lie(A_{\C_p}) \to \Lie(\Lie(A_{\C_p})) = \Lie(A_{\C_p})\]
is the identity.
  On the subgroup $A^1(\C_p)\subset A(\C_p)$ consisting of all points reducing to the
  origin, $\log_{A}$ is given by integrating
  expansions of holomorphic differentials into convergent power series. 
  Finally, we have
  $\log(A(K))\subset \Lie(A)\subset \Lie(A_{\C_p})$. 
\end{prop}
We call $\log_{A}$ the \textit{abelian logarithm} on $A$. 

\subsection{Abelian integration} The \textit{abelian integral} on $A$ is
defined as follows. For $\omega\in\Omega^1(A_{\C_p})$ and $P,Q\in A(\Cp)$, we set
\[ \presuper\Ab\int_P^Q\omega = \angles{\log_{A}(Q),\,\omega} - \angles{\log_{A}(P),\,\omega}\in \Cp\]
where $\angles{\cdot,\cdot}$ is the pairing between $\Lie(A)$ and
$\Omega^1(A)$. This integral enjoys a number of natural properties; for
example, it is additive in endpoints, linear in the integrand and
functorial under homomorphisms of abelian varieties; we have  
\begin{equation}\label{}
\presuper\Ab\int_P^Q\omega \in K
\end{equation}
for $P,Q\in A(K)$ and $\omega \in \Omega^1(A)$.

\subsection{Abelian integration on curves} Let $C/K$ be a nice curve with
Jacobian variety $J$. Let $j\colon
C\xhookrightarrow{} J$ be the Abel--Jacobi embedding relative to a fixed
base point $P_0\in C(K)$. It induces an isomorphism $j^*\colon
\Omega^1(J_{\C_p})\xrightarrow{\sim} \Omega^1(C_{\C_p})$ which restricts to
an isomorphism $j^*\colon
\Omega^1(J)\xrightarrow{\sim} \Omega^1(C)$ and is independent of the
chosen base point. For $\omega\in\Omega^1(C_{\C_p})$ and $P,Q\in C(\C_p)$, the
\textit{abelian integral} of $\omega$ from $P$ to $Q$ is defined as
\[ \presuper\Ab\int_P^Q\omega =
\presuper\Ab\int_{j(P)}^{j(Q)}(j^*)^{-1}\omega \in \C_p \,.\]
Once again, when all the data, namely $C$, $\omega$, $P$ and $Q$, are
defined over $K$, then 
$\presuper\lAb\int_P^Q\omega\in K$.

\begin{rk}\label{rem:TorsionVanishing}
	Because $\Lie(J)$ is torsion free, if $P,Q$ are points on $C$ with the
  property that $Q-P$ represents a torsion point, then
	\[\Abint_P^Q \omega = 0 \text{ for all } \omega\in \Omega^1(C).\]
	We will observe this vanishing numerically to test the correctness of our algorithm. 
\end{rk}

The abelian integral extends linearly to
an integral $\presuper\lAb\int_D$, where $D$ is a  divisor on $C$ of degree~0; this
extension is well-defined on $J$ by
construction.

\subsection{Existing algorithms for abelian integration on curves}
Abelian integration on curves of good reduction is the same as Coleman
integration \cite{Col85}. Such integrals can be explicitly computed using
the algorithms developed by Balakrishnan--Bradshaw--Kedlaya~\cite{BBK10},
Best~\cite{Bes21} and Balakrishnan--Tuitman \cite{BT20}. 
All these algorithms are restricted to the case where the base field is $\Q_p$ or a totally ramified extension of it. A generalization of~\cite{BT20} to curves over general $p$-adic fields is the subject of forthcoming work due to Cai and Keller.

In order to carry out abelian integration on hyperelliptic curves of bad
reduction, one can use the algorithms developed by Katz and Kaya \cite{KK22}; see also \cite{Kay22}. In these algorithms, the base field can be any $p$-adic field.

\subsection{Abelian logarithms and integrals on Jacobians of Mumford
curves}\label{LogMum}
Now let $C/K$ be a Mumford curve with Jacobian variety $J$. We show that the abelian logarithm can be expressed on the
uniformization $(K^\times)^g \to (K^\times)^g/Q_W\to J^{\an}(K) \simeq J(K)$ (see~\eqref{D1}).

We fix a base point $P_0\in C(K)$ such that $\infty \in \Omega_W(K)$ lifts
$P_0$ and obtain an isomorphism 
\begin{equation}\label{}
  \Lie(J)\simeq \Omega^1(J)^\vee\simeq \Omega^1(C)^\vee\simeq K^g\,, 
\end{equation}
where the final isomorphism is induced by a choice of basis of
$\Omega^1(C)^\vee$. We choose the basis $(\eta_1,\ldots,\eta_g)$
such that $\eta_i$ pulls back to the rigid analytic differential $\dlog
u_i$ on $\Omega_W$ (see also~\eqref{canemb}).
Under this identification, we are looking for a homomorphism $\widetilde{\log}\colon
(K^\times)^g \to K^g$ that induces a  homomorphism 
$\widetilde{\log}\colon
(K^\times)^g/Q_W \to K^g$ and makes the following diagram commute:
\begin{equation}\label{D2}
\begin{tikzpicture}[baseline={(current bounding box.center)}, >=stealth]
  \matrix (m) [matrix of math nodes, row sep=3em, column sep=4em]{
    \Omega_W(K) &  & T(K) && (K^\times)^g && \\
    \Omega_W(K)/W & & T(K)/\Lambda && (K^\times)^g/Q_W && \\
    C(K) & & J(K) &&&&  K^g \\};
  \draw[->>] (m-1-1) edge (m-2-1);
   \draw[->]  (m-1-1) edge node[below] {$j_W$}(m-1-3);
  \draw[->>] (m-1-5) edge (m-2-5);
   \draw[->]  (m-1-3) edge node[below] {$\varphi$} node[above] {$\simeq$} (m-1-5);
   \draw[->, dashed]  (m-1-5) edge node[above] {\;\;$\widetilde{\log}$} (m-3-7);
   \draw[->, dashed]  (m-2-5) edge node[below] {$\widetilde{\log}$} (m-3-7);
   \draw[->]  (m-2-3) edge node[below] {$\varphi$} node[above] {$\simeq$} (m-2-5);
   \draw[->]  (m-3-3) edge node[below] {$\log_{J}$} (m-3-7);
  \draw[->>] (m-1-3) edge (m-2-3);
   \draw[right hook-latex] (m-2-1) edge node[below] {$j_W$} (m-2-3);
   \draw[->] (m-2-1) edge node[left] {} (m-3-1);
   \draw[->] (m-2-3) edge node[left] {$\simeq$}  (m-3-3);
   \draw[->] (m-2-5) edge node[below] {$\simeq$}  (m-3-3);
   \draw[right hook-latex] (m-3-1) edge node[below] {$j$} (m-3-3);
\end{tikzpicture}
\end{equation}

We first look at the formal group $J^1(K)\subset J(K)$.

\begin{lemma}\label{L:formal}\hfill
  \begin{enumerate}
    \item 
      There is an open neighborhood  $U\subset (1+\m)^g\subset
      (K^\times)^g$ of $1$ such that the
restriction of $\widetilde{\log}$ to $U$
  is given by 
  \begin{equation}\label{}
    \widetilde{\log}\,\vec{t} = \log \vec{t}\colonequals (\log t_1,\ldots,\log t_g)\,.
  \end{equation}
\item We have 
  \begin{equation}\label{}
    \Abint_P^Q\eta_i = \log u_i(z_P) - \log u_i(z_Q)
  \end{equation}
      for all $P,Q\in C(K)$ in the same residue disc with respective lifts
      $z_P,z_Q\in \Omega_W(K)$, and $i\in \{1,\ldots,g\}$.
  \end{enumerate}
\end{lemma}
\begin{proof}
  We take $U$ to be the lift of $J^1(K)$. Since $P-Q$ belongs to $J^1(K)$ for $P,Q$ in
  the same residue disc, the result follows from
  Proposition~\ref{P:Zarhin}.
\end{proof}
In order to extend to all of $J(K)$, we fix any branch $\log\colon K^\times
\to K$ of the logarithm. However, the coordinates of $\widetilde{\log}$
might not be exactly $\log$
on all of $(K^\times)^g$ since the resulting map might not be well-defined on
$(K^\times)^g/Q_W$. 
This is similar to the situation of a Tate curve $E/K$ with $E(K)\simeq
K^*/q_E^{\Z}$, where one chooses a branch of the logarithm that vanishes on
the Tate parameter $q_E\in \m$ by correcting the chosen branch $\log$ of the logarithm 
by $\mathcal{L}\cdot \ord$, where $\mathcal{L}=\log(q_E)/\ord(q_E)$ is the
$\mathcal{L}$-invariant of $E$.

We extend this by defining a naive extension of the
$\mathcal{L}$-invariant. Namely, since $C=C_W$ is a Mumford curve, the matrix
$\ord(Q_W)\in \Q^{g\times g}$ given by taking the orders of all entries is
positive definite by~\cite[VI (2.4)]{GP80} and we may define a matrix
\begin{equation}\label{Linvmat}
  \mathcal{L} \colonequals \log(Q_W) \cdot \ord(Q_W)^{-1}\in K^{g\times
  g}\,,
\end{equation}
where $\log(Q_W)\in K^{g\times g}$ is defined by taking the logarithms of all
entries.
\begin{lemma}\label{L:extend}
  There is a unique continuous homomorphism $\widetilde{\log}\colon
  (K^\times)^g \to K^g$  that
  restricts to $\log$ on $U$ and vanishes in the multiplicative
  lattice generated by the columns of $Q_W$; it is given by
  \begin{equation}\label{logformula}
    \widetilde\log\,\vec{t} = \log \vec{t} - \mathcal{L}\cdot \ord \vec{t}\,,
  \end{equation}
  where $\log \vec{t} = (\log t_1,\ldots,\log t_g)$ and
$\ord \vec{t} = (\ord t_1,\ldots,\ord t_g)$.
\end{lemma}
\begin{proof}
  The coordinates of a continuous homomorphism $(K^\times)^g \to K^g$ that
  restricts to $\log$ on $(1+\m)^g$ must be given by branches of the
  logarithm; in other words there are constants $c_1,\ldots, c_g\in K$ such
  that
  \begin{equation*}
    \widetilde{\log}_i(\vec{t}) = \log t_i + c_i\ord t_i\,. 
  \end{equation*}
  Hence, if there is a homomorphism $\widetilde{\log}$ with the desired
  properties, then it must be given by~\eqref{logformula}.
  Since $\log$ and $\ord$ are both homomorphisms,~\eqref{logformula} indeed
  satisfies these properties.
\end{proof}
Hence we have shown
\begin{thm}\label{T:AbIntMum}
  Let $D\in \Div^0(C)$ pull back to $\tilde D\in \Div^0\Omega_W(K)$. Then we have 
  \begin{equation}\label{}
    \left(\Abint_D\eta_1,\ldots,\Abint_D\eta_g\right)  = \left(\log
     u_1(\tilde D), \ldots, \log u_g(\tilde D)\right)
    -\mathcal{L}\cdot \left( \ord
    u_1(\tilde D),\ldots,\ord u_g(\tilde D)\right)\,.
  \end{equation}
\end{thm}

\subsection{Our algorithm}\label{}
Our algorithms allow us to compute abelian integrals on hyperelliptic Mumford curves using Theorem~\ref{T:AbIntMum}:

\begin{algorithm}[H] \label{A:AbelianIntegrals}
	\caption{\bf Computing Abelian integrals  on hyperelliptic Mumford curves}
	\KwIn{
		\begin{itemize}
			\item A hyperelliptic Mumford $C/K$ 	uniformized by a Whittaker group $W$ with generators $w_1,\dots,w_g$.
			\item A degree~$0$ divisor $D$ on $C$.
		\end{itemize}
		}
	\KwOut{The Abelian integrals $\lAbint_D\eta_1,\ldots,\lAbint_D\eta_g$ where $\eta_i = \dfrac{u'_{i}(z)}{u_{i}(z)}dz$ and $u_{i} = u_{w_i}$.}
	\begin{enumerate}
		\item Compute a lift $\tilde D$ of $D$ to $\Omega_W(K)$.
		\item Compute the period matrix $Q_W$ and the $\mathcal{L}$-invariant $\mathcal{L} = \log(Q_W) \cdot \ord(Q_W)^{-1}$.
		\item Compute the tuple $\big(u_1(\tilde D), \ldots, u_g(\tilde D)\big)$, and $\log$ and $\ord$ of its components.
		\item Return
		\[\left(\Abint_D\eta_1,\ldots,\Abint_D\eta_g\right)  = \left(\log
		u_1(\tilde D), \ldots, \log u_g(\tilde D)\right)
		-\mathcal{L}\cdot \left( \ord
		u_1(\tilde D),\ldots,\ord u_g(\tilde D)\right).\] 
	\end{enumerate}
\end{algorithm}

\section{Schneider $p$-adic heights}\label{S:heights}

In analogy with the real-valued N\'eron--Tate height pairing on an abelian
variety $A$ over a number field $F$, various authors have defined symmetric
bilinear $p$-adic height pairings $A(F)\times A(F) \to \Q_p$. Explicit
methods for computing $p$-adic height pairings have been instrumental in
two versions of 
the quadratic Chabauty method for computing rational points on curves,
see~\cite{BD18} and~\cite{BMS25}. They are also important in order to
numerically check a $p$-adic version of the conjecture of Birch and
Swinnerton--Dyer in examples; see~\cite{mtt, BMS16}.

All $p$-adic height pairings in the literature can be decomposed into sums
of local height pairings, one for each non-archimedean place of $F$. 
They also require the choice 
of a nontrivial continuous $\Q_p$-valued id\`ele class character $\chi$ on $F$. 
In the present article, we focus on the $p$-adic height pairing
\begin{equation}\label{}
  \langle\cdot, \cdot\rangle_{\chi}\colon J(F)\times J(F)\to \Q_p
\end{equation}
constructed by Schneider
\cite{Sch82}, where  the abelian variety in question is the Jacobian
of a curve $C/F$. 
\begin{rk}\label{R:}
  Schneider's construction relies on an assumption which, in the setting we
  are concerned with, translates into Assumption~\ref{ass2} below.
\end{rk}

We fix a finite place $v$ of $F$ and we denote $K\colonequals F_v$.
The component $\chi_v$ is then a non-trivial continuous homomorphism
\begin{equation}\label{}
\chi_v=\rho \colon K^{\times}\to \Q_p\,.
\end{equation}
In this case, as shown by Werner~\cite{Wer96},
the local $p$-adic Schneider height pairing $\langle D,E\rangle_\rho$ on $C/K$ 
with respect to $\rho$ can be written as a symmetric
biadditive pairing on divisors of degree~0 on $C/K$ with disjoint support. If
$D=\div(f)$ is principal, then
\begin{equation}\label{princ}
  \langle D,E\rangle_\rho = \rho(f(E))\,.
\end{equation}
If $L/K$ is a finite extension
and $D,E\in \Div^0(C)$ have disjoint support,
then we have
\begin{equation}\label{htext}
\langle D\otimes L,E\otimes L\rangle_{\rho_L} = \langle
D,E\rangle_{\rho}\cdot [L:K]\,,
\end{equation}
where $\rho_L = \rho\circ N_{L/K} \colon
  L^\times\to \Q_p$. The global Schneider height pairing then satisfies
\begin{equation}\label{}
  \langle P,Q\rangle_{\chi} = \sum_v\langle D_v,E_v\rangle_{\chi_v}
\end{equation}
for any $P,Q\in J(F)$ and $D,E\in \Div^0(C)$ with disjoint support representing $P,Q$,
respectively, where we write 
$D_v = D\otimes F_v$ and $E_v = E\otimes F_v$. The \textit{Schneider regulator (with respect to $\chi$)} is defined by
\begin{equation}\label{reg}
	\reg_{\chi}(J/F) \colonequals \reg_{\chi}(P_1,\ldots,
	P_r)\colonequals (\langle P_i,P_j\rangle_{\chi})_{1\le i,j\le
		r}\,,
\end{equation}
where the classes of $P_1,\ldots,P_r$ form a basis of $J(F)$ modulo torsion.
Schneider has conjectured in~\cite{Sch82} that the Schneider regulator with respect to the
cyclotomic character is nonzero.

\begin{rk}\label{R:}
  Mazur and Tate construct height pairings in great generality
  in~\cite{MT83}; in particular, they show that Schneider's height pairing
  is a special case of their construction. For abelian varieties with
  ordinary reduction at all places above $p$, they also define a canonical
  $p$-adic height pairing. 
By \cite[\S4.4]{MT83}, the canonical Mazur--Tate height pairing is
  equal to the Schneider height pairing for abelian varieties with good
  ordinary reduction. These two pairings, however, differ in general, as
  noted by Mazur--Tate--Teitelbaum (see~\cite[\S I.6]{mtt}): a formula for
  the difference in the case of semistable ordinary reduction was given by
  Werner \cite[Theorem~7.2]{Wer98}. 
\end{rk}

We now define the local height pairing $\langle \cdot,\cdot\rangle_{\rho}$ by distinguishing cases when the character $\rho$ is unramified or ramified. Recall that $\rho$ is called unramified if $\rho(\O_v^\times)=0$.

\subsection{The unramified case}\label{subsec:unram}

 In this case, the local height pairing can be computed in terms of
arithmetic intersection theory; up to a constant multiplicative factor, it
is the same for all definitions of $p$-adic height
pairings between divisors of degree~0 with disjoint support, and in fact
coincides with the real-valued local N\'eron height pairing at $v$ defined
in~\cite[\S III.5]{Lan88}.
It is given in terms of intersection theory on a proper regular model
$\calC/\O_K$ of $C/K$ by
\begin{equation}\label{IntersThryForm}
(D,E)_\rho = \rho(\pi_K)\cdot(\calD\cdot\calE)\in \Q_p\,,
\end{equation}
where $\pi_K$ is a uniformizer of $\O_K$, and $\calD$ and $\calE$ are extensions of
$D$ and $E$, respectively, to $\calC$ that have
trivial intersection multiplicity with all vertical divisors.
See~\cite{Hol12, Mul14} for algorithms to compute the local height pairing
for hyperelliptic curves, and~\cite{vBHM20} for general curves.

\subsection{The ramified case}\label{subsec:ram}
Now suppose that the continuous homorphimsm $\rho$ is ramified; in
particular, we have $v\mid p$.
Let $C/K$ be a Mumford curve of genus $g$. Then $C = \Omega_W/W$ for a
Schottky group $W$ of rank $g\ge 2$ (which
is only Whittaker when $C$ is hyperelliptic). We write $\Omega=\Omega_W$
and fix good generators $w_1,\ldots,w_g$ of $W$.
\begin{ass}\label{ass1}
We assume that the field $K$ is large enough so that
\begin{enumerate}
  \item $\Omega(K)\ne \varnothing$, and
  \item the entries of the period matrix $Q\colonequals Q_W\in K^{g\times g}$
    from~\S\ref{subsec:periods} are squares in $K^\times$.
\end{enumerate}
\end{ass}
Assumption~\ref{ass1} can be achieved, if necessary, by passing to a finite
extension.
We will also suppose that the following holds:
\begin{ass}\label{ass2}
  The matrix $\rho(Q) = \big(\rho(q_{ij})\big)_{i,j}$ is invertible, where
  $Q=(q_{ij})_{i,j}$.
\end{ass}
Assumption~\ref{ass2} can be checked easily in practice, once we have
computed the period matrix.
Loosely speaking, this condition says that the image of the lattice $\Lambda$ under the map $\rho$ is a lattice of full rank in $\Q_p^g$.
We have not found an example where $\rho(Q)$ is not invertible.
Thanks to \cite[Proposition~4.12]{Wer97}, this assumption implies that the condition for the existence of Schneider’s local $p$-adic
height pairing on $J$ is fulfilled. 

By Assumption~\ref{ass1} we may fix, for all $i,j\in \{1,\ldots,g\}$,
$p_{ij}\in K^\times$
        such that 
        $$p_{ij}^2 = Q_{ij} =\langle [w_i], [w_j]\rangle_{W^{\ab}} =( \iota(w_i),\iota(w_j))_W\quad\text{ and}\quad p_{ij} =
        p_{ji}\,.$$ 
        We call the $p_{ij}$~\textit{half-periods} and we call the matrix
        \begin{equation}\label{}
          P \colonequals (\rho(p_{ij}))_{ij} \in K^{g\times g}
        \end{equation}
        an~\textit{additive half-period matrix}. 

For $D'\in \Div^0(\Omega)_K$, the group of divisors of degree~0 on $\Omega$
that are rational over $K$, we define
\begin{equation}\label{xi}
  \xi(D') \colonequals \bigg(\rho\Big(\big(\iota(w_i),D'\big)_W\Big)\bigg)_{i}\in K^g\,,
\end{equation}
where we use Remark~\ref{R:WeilRec}.
Let $[\cdot,\cdot]_P$ be the symmetric bilinear form attached to $P$.

\begin{thm}\label{T:Werner}\emph{(Werner~\cite[Theorem~3.2]{Wer96})}
  Let $D,E\in \Div^0(C_K)$ have disjoint support and choose lifts $D',E'\in
  \Div^0(\Omega)_K$ of $D,E$, respectively.
  Then the local Schneider height pairing satisfies
  $$\langle D,E\rangle _\rho =  \rho\big((D',E')_W\big)-
  [\xi(E'),\xi(D')]_P\,.$$
\end{thm}

\begin{rk}\label{R:}
  Werner only states and proves Theorem~\ref{T:Werner} for differences
  $D,E$ of pairs of points in $C(K)$. The extension to arbitrary $D,E\in Z^0(C_K)$ is
  immediate, and the full statement follows from Remark~\ref{R:WeilRec},
  using that the local Schneider height behaves well with respect to field
  extensions by~\eqref{htext}. 
\end{rk}

\subsection{Computing Schneider heights}\label{}
Suppose that $C/K$ is a hyperelliptic Mumford curve uniformized by a
Whittaker group $W$, and that $\rho\colon
K^\times \to \Q_p$ is ramified.
In our examples, $\rho$ will typically be a local component of the cyclotomic
character, so on $K=\Q_p$
it will be the Iwasawa branch $\log$ of the $p$-adic log, determined by $\log(p)=0$. We will
usually have to work over a proper extension $K/\Q_p$, in which case we use 
$$\rho(x)
= \log N_{K/\Q_p}(x)\,.$$
However, our implementation allows more general homomorphisms.

\begin{algorithm}[H] \label{A:SchneiderHeights}
	\caption{\bf Computing local Schneider heights on hyperelliptic Mumford curves}
	\KwIn{
		\begin{itemize}
			\item A hyperelliptic Mumford $C$ over $K$.
			\item A continuous homomorphism $\rho \colon K^{\times}\to \Q_p$ that is ramified.
			\item $D,E\in \Div^0(C)$ with disjoint support.
		\end{itemize}
	}
	\KwOut{The local Schneider height pairing $\langle D,E\rangle_\rho$.}
	\begin{enumerate}
		\item Compute a Whittaker group $W$ uniformizing $C$ using the
      algorithms in Section~\ref{S:FixedFromRam}.
		\item Compute the period matrix $Q=Q_W$. If $\rho(Q)$ is not invertible, then 
      throw an error.
		\item Compute an additive half-period matrix $P$ by extending the base field if necessary.
		\item Compute lifts $D',E'\in
		\Div^0(\Omega)_K$ of $D,E$ using Algorithm~\ref{A:LiftPoints}.
		\item Compute $(D',E')_W,\, \xi(D')$ and $\xi(E')$.
		\item Return $\langle D,E\rangle _\rho =  \rho\big((D',E')_W\big)-
		[\xi(E'),\xi(D')]_P$.
	\end{enumerate}
\end{algorithm}

This allows us to compute $\langle D,E\rangle_{\rho}$ to any desired
$p$-adic precision. We use the naive algorithm to compute
$\theta_W$ from~\S\ref{subsec:naive} for the first few digits, and then we switch to the
iterative algorithm from~\cite{MX25} outlined
in~\S\ref{subsec:iterative}.
\subsubsection{Computing global Schneider heights}\label{AlgoGlobal}
Suppose that $F$ is a number field and that $C/F_v$ is a hyperelliptic
Mumford curve for all $v\mid p$ such that $\chi_v$ is ramified. 
In order to compute global Schneider heights between two points $P,Q\in
J(F)$, we first need to find nonzero integers $n,m$ and divisors $D,E\in \Div^0(C)$ with disjoint
support such that $D$ represents $nP$ and $E$ represents $mQ$; 
the bilinearity of the Schneider height implies
\begin{equation}\label{}
  \langle P,Q\rangle = \frac{1}{nm}\sum_v \langle D_v,
  E_v\rangle_{\chi_v}\,,
\end{equation}

For our algorithms, we try to find $D,E,n,m$ so that $D_v, E_v \in
Z^0(\Omega_{W_v}(F_v))$ for all $v$ such that $\chi_v$ is ramified, where
$W_v$ is a Whittaker group uniformizing $C_v$.

We either use the built-in {\tt Magma}-command {\tt LocalIntersectionData}, which
implements the algorithm from~\cite{Mul14} or the {\tt Magma}-implementation of the
more flexible algorithm
from~\cite{vBHM20}  available from
{\url{https://github.com/emresertoz/neron-tate}} to compute $\langle D_v,
  E_v\rangle_{\chi_v}$ when $\chi_v$ is unramified. The
  ramified pairings are computed via Algorithm~\ref{A:SchneiderHeights}.

 \subsection{$p$-adic BSD}\label{S:BSD}

A $p$-adic analogue of the Birch and Swinnerton--Dyer conjecture for
an elliptic curve over $\Q$ was given by Mazur--Tate--Teitelbaum 
\cite{mtt} when $p$ is a prime of good ordinary or multiplicative
reduction, with the canonical regulator defined in terms of the
N\'eron--Tate height replaced by the Schneider regulator (with respect to
the cyclotomic character). 
Balakrishnan--M{\"u}ller--Stein \cite{BMS16} formulated a
generalization of the Mazur--Tate--Teitelbaum conjecture in the good ordinary case to
higher-dimensional modular abelian varieties of $\GL_2$-type over $\Q$.
They also provided numerical evidence supporting their conjecture for
Jacobians of genus 2 curves. To compute the regulator, they used that
Schneider's height is equivalent to the Coleman--Gross height with respect
to the unit root subspace and computed the latter using algorithms for
Coleman integration (see~\cite{BB12}). This height is expected to be equivalent to the
canonical Mazur--Tate height (see~\cite{BB15, BKM} for proofs for
dimension~1 and~2, respectively), and hence different from the Schneider height
in bad reduction.

On the other hand, the Mazur--Tate--Teitelbaum conjecture in the case of split multiplicative
reduction, the \textit{exceptional} case, is of special interest. One might
expect that a generalization of this conjecture to higher-dimensional
modular abelian varieties of $\GL_2$-type over $\Q$ in the case of split
purely toric reduction can be formulated. Formulating such a conjecture, as
well as gathering numerical evidence for it, requires the computation of Schneider heights. To that end, the findings of the current project can be used.

\section{Numerical examples}\label{sec:examples}
We applied the algorithms described in this paper to various examples of
low genus hyperelliptic Mumford curves. The files listed in this section
can all be found in our repository
{\url{https://github.com/mmasdeu/hyperellipticmumford}}.

\begin{ex}\label{Ex:KadExamples}
	
	In the final part of his thesis, Kadziela computed
	for $6$ different hyperelliptic Mumford curves of genus $2$ over $\Q_5$ 
	the corresponding Whittaker groups and period matrices using {\tt Magma}; see
	\cite[\S7.4]{Kad07}. Unfortunately, we were not able to find his code.
	More importantly, the fixed points in his examples are not in strong
	Kadziela
	position, even though this is necessary for the correctness of his
	algorithm. However, the fixed points turn out to be in weak Kadziela
	position. Hence, using our extension of Kadziela's algorithm
	in \S\ref{S:Kadziela} based on the approximation result
	Proposition~\ref{P:approxg2}, we managed to verify his claims. The
	complete computation can be found in the file
	\texttt{KadzielaExamples.sage}. In particular, for the curve $X_1$ in \cite[\S7.4]{Kad07} which is given by the equation  
	\[y^2 = x^5 - 326x^4 + 1052\cdot5^2x^3 - 5914\cdot5^2x^2 + 39\cdot5^5x = x(x-5^3)(x-5)(x-195)(x-1),\]
	our computations show that the nontrivial fixed points are
	\[b_0 = 5^{3} \cdot 87495069218 + O(5^{20}),\ \ \ a_1 = {5 \cdot 7806971503561 + O(5^{20})},\ \ \ b_1 = {5 \cdot 12203741012063 + O(5^{20})},\]
	and that the matrices 
	\begin{small}
		\[w_1 = \left(\begin{array}{rr}
			5 \cdot 937226187499 & 5 \cdot 1866057579057  \\
			2 & 4686130937491
		\end{array}\right) + O(5^{20}), \ \ \ w_2 = \left(\begin{array}{rr}
			5^{3} \cdot 87495069218 & 5^{3} \cdot 587949314689 \\
			2 & 10936883652246
		\end{array}\right) + O(5^{20})\]
	\end{small}
	generate a Whittaker group $W$ that uniformizes $X_1$. Finally, the corresponding period matrix is
	\[Q_W = 
	\left(\begin{array}{rr}
		5^{2} \cdot 825167307749 & 5^{2} \cdot 1908298341511 \\
		5^{2} \cdot 1908298341511 & 5^{6} \cdot 787156051
	\end{array}\right) + O(5^{20}).\]
	
\end{ex}

\begin{ex}\label{E:AbEx} Consider the hyperelliptic curve $C/\Q$ \cite[\href{https://www.lmfdb.org/Genus2Curve/Q/3950/b/39500/1}{3950.b.39500.1}]{LMFDB} given by
\[y^2 = (x^2-x-1)(x^4+x^3-6x^2+5x-5).\]
  It is a Mumford curve over $\Q_5$ 
  and the stable reduction is the union of two projective lines meeting transversally at three points:
  \[\scalebox{0.7}{\begin{tikzpicture}
	\draw[thick] (0,-0.4) to[bend left] (4,0);
	\draw[thick] (4,0) to[bend right] (8,0.4);
	\draw[thick] (0,0.4) to[bend right] (4,0);
	\draw[thick] (4,0) to[bend left] (8,-0.4);
  \end{tikzpicture}}\]
In this example, in order to test the correctness of our algorithms, we do two experiments:

\begin{itemize}
	\item \textbf{Test I:} According to the database, the Mordell-Weil group
    of the Jacobian of $C$ over $\Q$ is isomorphic to $\Z/12\Z$. Therefore,
    for any two points $P,Q\in C(\Q)$, the Abelian integrals against the
    divisor $(Q)-(P)$ must vanish; see Remark~\ref{rem:TorsionVanishing}. We will confirm this for points $P=(1,2),\ Q=(1,-2)\in C(\Q)$.
	\item \textbf{Test II:} Certain Abelian integrals on this curve are already computed in \cite[Example~7.2]{Kay22}, and the code used in the computations allows us to deal with all Abelian integrals\footnote{In \emph{loc. cit.}, the notion of ``Vologodsky" integration is discussed, of which abelian integration is a special case. More precisely, Vologodsky integrals of holomorphic forms are precisely abelian integrals.}. Using that we see that
	\begin{equation}\label{eq:Abel_ints}
	\left(\Abint_P^Q\dfrac{dx}{2y},\Abint_P^Q x\dfrac{dx}{2y}\right)	 = \big(5 \cdot 4457788593 + O(5^{15}),\, 5^{2} \cdot 53059578 + O(5^{15}) \big)
	\end{equation}
	where $P = \big(-1,62455561379272 + O(5^{20})\big)$ and $Q = \big(6,29140889570072 + O(5^{20})\big)$.
	We will compute these integrals using our methods and compare the results.
\end{itemize}

We first compute the fixed points and generators of a Whittaker group that
  parameterizes our curve, and the corresponding period matrix. We will
  work within the field $K = \Q_5(\pi)$ where $\pi$ satisfies $\pi^4$ = 5.
  Our code gives that the nontrivial fixed points are
\begin{align*}
	b_0 &= \pi^{4} \cdot (147218987704020546367 + 76517089843750000000 \cdot \pi^{2}) + O(\pi^{120}), \\
	a_1 &= \pi^{2} \cdot (270181800544006476117 + 677609166444683833921 \cdot \pi^{2}) + O(\pi^{120}), \\
	b_1 &= \pi^{2} \cdot (866164628265668733832 + 653861286121912906612 \cdot \pi^{2}) + O(\pi^{120})	
\end{align*}
and that the matrices
\begin{small}
\[
w_1 = 
\left(\begin{array}{rr}
	\pi^{-8} \cdot (566260181145554 + 507354736328125 \cdot \pi^{2}) & \pi^{-4} \cdot (28220976712906 + 328063964843750 \cdot \pi^{2}) \\
	\pi^{-8} \cdot (1439946765189457 + 2227783203125000 \cdot \pi^{2}) & \pi^{-4} \cdot (14110488356453 + 164031982421875 \cdot \pi^{2})
\end{array}\right) + O(\pi^{80})
\]
\[
w_2 = 
\left(\begin{array}{rr}
	\pi^{-4} \cdot (384303682828926 + 241352744313936 \cdot \pi^{2}) & 92591457234474 + 16807499086906 \cdot \pi^{2} \\
	\pi^{-6} \cdot (154765152323318 + 1334338488576612 \cdot \pi^{2}) & \pi^{-2} \cdot (23876747388397 + 257490127521949 \cdot \pi^{2}) 
\end{array}\right) + O(\pi^{80})
\]
\end{small}
generate the corresponding Whittaker group $W$. Moreover, the corresponding period matrix is
\[
Q_W = 
\left(\begin{array}{rr}
	5^{2} \cdot 23054690584 + O(5^{17}) & 5 \cdot 6733107769 + O(5^{16}) \\
	5 \cdot 6733107769 + O(5^{16}) & 5^{2} \cdot 11569554284 + O(5^{17})
\end{array}\right).
\]

Now we can conduct our experiments:
\begin{itemize}
\item \textbf{Test I:} Let $D = (Q)-(P)$ where $P=(1,2)$ and $Q=(1,-2)\in C(\Q)$. The lift of this divisor is
\begin{align*}
	D' &= \Big(\pi^{4} \cdot (161517344335908068786 + 84879185053106998144 \cdot \pi^{2})\Big) \\
	& \ \ \ - \Big(\pi^{2} \cdot (205078598150168020669 + 147110824295460735763 \cdot \pi^{2})\Big) + O(\pi^{120}).
\end{align*}
The $\mathcal{L}$-invariant is
\begin{align*}
	\mathcal{L} = \log(Q_W) \cdot \ord(Q_W)^{-1} &= 
	\left(\begin{array}{rr}
		5 \cdot 861817323 + O(5^{15}) & 5 \cdot 3360896081 + O(5^{15}) \\
		5 \cdot 3360896081 + O(5^{15}) & 5 \cdot 1091212433 + O(5^{15})
	\end{array}\right)\cdot \left(\begin{array}{rr}
		2 & 1 \\
		1 & 2
	\end{array}\right)^{-1} \\
	&= \left(\begin{array}{rr}
		5^{2} \cdot 704651321 + O(5^{15}) & 5 \cdot 6022335363 + O(5^{15}) \\
		5 \cdot 3911365118 + O(5^{15}) & 5^{2} \cdot 328336294 + O(5^{15})
	\end{array}\right).
\end{align*}
The images of this lift under the functions $u_1$ and $u_2$ are
\[\left(u_1(\tilde D),u_2(\tilde D)\right) = \big(5 \cdot 18104738516 + O(5^{16}),\,5 \cdot 19143541241 + O(5^{16})\big).\]
Finally, using the formula in Theorem~\ref{T:AbIntMum}, we see that
\begin{align*}
	\left(\Abint_D\eta_1,\Abint_D\eta_2\right) &= \left(\log
	u_1(\tilde D), \log u_2(\tilde D)\right)
	-\mathcal{L}\cdot \left( \ord
	u_1(\tilde D),\ord u_2(\tilde D)\right) \\
	&= \big(5 \cdot 3442076343 + O(5^{15}),\,5 \cdot 5553046588 + O(5^{15})\big) - \mathcal{L}\cdot \big(1,\,1\big) = \big(O(5^{15}),\,O(5^{15})\big).
\end{align*}

\item \textbf{Test II:} We first determine the change of basis matrix, i.e., the matrix $M \in \GL_2(\Q_5)$ with the property that
\begin{equation}\label{eq:ChangeOfBasis}
\left(
\begin{array}{c}
	\dfrac{dx}{2y} \\
	x\dfrac{dx}{2y}
\end{array}
\right)
= M\cdot
\left(
\begin{array}{c}
	\eta_1 \\
	\eta_2
\end{array}
\right).
\end{equation}
Consider the degree~$0$ divisors $D_P = (P')-(P)$ and $D_Q = (Q')-(Q)$ on $C_{\Q_5}$ where $P' = \big(4, 10618881637377 + O(5^{20})\big)$ and $Q' = \big(11, 84476472732667 + O(5^{20})\big)$.
The equality~\eqref{eq:ChangeOfBasis} implies that we should have
\[
\left(
\begin{array}{cc}
	\displaystyle\Abint_{D_P}\dfrac{dx}{2y} & \displaystyle\Abint_{D_Q}\dfrac{dx}{2y} \\
	\displaystyle\Abint_{D_P}x\dfrac{dx}{2y} & \displaystyle\Abint_{D_Q}x\dfrac{dx}{2y}
\end{array}
\right)
= M\cdot
\left(
\begin{array}{cc}
	\displaystyle\Abint_{D_P}\eta_1 & \displaystyle\Abint_{D_Q}\eta_1 \\
	\displaystyle\Abint_{D_P}\eta_2 & \displaystyle\Abint_{D_Q}\eta_2
\end{array}
\right).
\] 
Algorithm~\eqref{A:AbelianIntegrals} allows us to compute the integrals on
the right-hand side. On the other hand, the integrals on the left-hand side
are integrals between points in the same residue class, and hence can be
easily computed as ``tiny'' integrals, as described in
\cite[Algorithms~1~\&~3]{Bal15}\footnote{In \emph{loc. cit.}, the notion of
``Coleman" integration is discussed, which is the
same as Abelian integration in the case of good reduction. 
Since the points $P,P',Q,Q'$ reduce to smooth points of the special fiber under the reduction map, \cite[Algorithms~1~\&~3]{Bal15} apply verbatim to our setting as well.}. Our computations give
\[M = 
\left(\begin{array}{rr}
	4618661532 + O(5^{14}) & 2491321374 + O(5^{14}) \\
	5^{2} \cdot 232513793 + O(5^{14}) & 1879285442 + O(5^{14})
\end{array}\right).\]

We can now compute the desired integrals. Let $D = (Q)-(P)$. Using Theorem~\ref{T:AbIntMum} once again, we see that 
\[
\left(\Abint_D\eta_1,\Abint_D\eta_2\right)
= \big(5 \cdot 3734524359 + O(5^{15}),\, 5^{2} \cdot 1102748129 + O(5^{15}) \big).
\]
Combining this with~\eqref{eq:ChangeOfBasis}, we get
\[
\left(\Abint_D\dfrac{dx}{2y},\Abint_Dx\dfrac{dx}{2y}\right)	 = \big(5 \cdot
4457788593 + O(5^{15}),\, 5^{2} \cdot 53059578 + O(5^{15}) \big)\,,
\]
which is the same as Equation~\eqref{eq:Abel_ints}.
\end{itemize}

 The complete computation can be found in the files
\texttt{AbelianIntegrals.sage} and \texttt{AbelianIntegralsSupplement.sage}.

\end{ex}

\begin{ex}\label{E:}
  Let $C/\Q$ be the genus~2 curve whose Jacobian $J$ is the modular abelian
  variety associated to the newform
  \cite[\href{https://www.lmfdb.org/ModularForm/GL2/Q/holomorphic/145/2/a/b/}{145.2.a.b}]{LMFDB}.
  We thank Andrew Sutherland for sharing this example (and many
  others) from his unpublished database of genus~2 curves with us.
  We start with the equation
  \[
y^2 = -351x^6 + 918x^5 - 837x^4 + 394x^3 - 119x^2 + 20x\,.
  \]
  The curve $C$ is a Mumford curve over $\Q_5$ of reduction type (a). We
  first find a model of $C$ pver $\Q_5$ in strong Kadziela position and we
  compute the fixed points and good generators of the corresponding
  Whittaker group via the methods in Section~\ref{S:FixedFromRam}. We then
  apply Algorithm~\ref{A:SchneiderHeights}
  to compute local (and, as in~\S\ref{AlgoGlobal}, also global) Schneider heights for the following
  divisors
  \begin{align*}
    D &= \left(\frac{1}{12}, \frac{1755}{12^3}\right) - \left(0,0\right)\\
    E &= \left(\frac{1}{3},1\right) - \left(\frac{5}{6},\frac{675}{6^3}\right)\\
    \div\left(\frac{x-1}{2x-1}\right) &= \left(1,5\right)+\left(1,-5\right) -
    \left(\frac12,\frac58\right) - \left(\frac12,
    -\frac58\right)\,.
  \end{align*}
  Let $\rho\colon \Q_5^\times \to \Q_5$ be the Iwasawa branch of the
  logarithm.
We verify numerically that 
  \begin{equation}\label{}
    (D,E)_\rho = (E,D)_\rho = 5\cdot 80537980048677 + O(5^{21})
 \,.
  \end{equation}
This is a nontrivial check, see the formula in Theorem~\ref{T:Werner}. Moreover we
  check that 
  $$(\div(f),E)_\rho = \rho(f(E)) = 5\cdot 91879536598576 + O(5^{21})\,,$$ thus verifying~\eqref{princ}.
    
    Finally, we compute the global Schneider regulator $\reg_{\chi}(J/\Q))$
    defined in~\eqref{reg}, where $\chi$ is the cyclotomic character.
    {\tt Magma}'s~{\tt{MordellWeilGroupGenus2}} shows that
    the Mordell--Weil rank of $J/\Q$ is~2 and that
    $J(\Q)/J(\Q)_{\mathrm{tors}}$ is generated by the classes of the points
    with Mumford representation
    \begin{equation}\label{}
      P_1 = (x^2 - 1/3x, y + 3x),\ \ \ P_2= (x^2 - 5/6x + 1/6, y + 9/4x - 7/4)\,.
    \end{equation}
    We find that the regulator is 
    \begin{equation}\label{}
      \reg_{\chi}(J/\Q)) = \reg_{\chi}(P_1,P_2) = 5^2\cdot 85823815506491 +
      O(5^{22}) \,.
    \end{equation}
As a sanity check, we also compute the regulator
$\reg_{\chi}(Q_1,Q_2)$
    where 
    \begin{equation}\label{}
      Q_1 = (x^2 - 2/3x + 1/9, 2x - 5/3),\ \ \ 
Q_2 = (x^2 - 11/15x + 1/30, -189/50x + 26/25)\,.
    \end{equation}
    which generate a subgroup of index~4, and we find that 
    \begin{equation}\label{}
      \reg_{\chi}(Q_1,Q_2) = 16\cdot \reg_{\chi}(J/\Q)\,, 
    \end{equation}
    as predicted by the quadraticity of the global Schneider height.

  The code for this example can be found in the file
  \texttt{145ab.sage}.
  
\end{ex}

\begin{ex}\label{x039}
	Consider the modular curve $X_0(39)/\Q$. This is a hyperelliptic curve of genus~$3$ with equation
	\[y^2 = (x^4 - 7x^3 + 11x^2 - 7x + 1)(x^4 + x^3 - x^2 + x + 1);\]
	see, for example, \cite[Table~4]{Gal96}. The roots of the defining polynomial lie in the field $K = \Q_3(i,\pi)$ where $i^2 = -1$ and $\pi^2 = 3$. Moreover, over this field, the given curve becomes a Mumford curve and the corresponding (stable) reduction is the union of two projective lines meeting transversally at four points:
	\[\scalebox{0.7}{\begin{tikzpicture}
		\draw[thick] (0,-0.4) to[bend left] (4,0);
		\draw[thick] (4,0) to[bend right] (8,0);
		\draw[thick] (8,0) to[bend left] (12,-0.4);
		
		\draw[thick] (0,0.4) to[bend right] (4,0);
		\draw[thick] (4,0) to[bend left] (8,0);
		\draw[thick] (8,0) to[bend right] (12,0.4);
	\end{tikzpicture}}\]
	
	The Jacobian $J_0(39)/\Q$ satisfies $J_0(39)(\Q)\simeq \Z/28\Z$). The curve has four cusps; these are
	$(0,\pm1)$ and $\infty_{\pm}$ on our model. The
	latter are the points $(1:\pm1:0)$ on the closure of our model in the
	projectve plane with respective weights $1,4,1$ attached to $X,Y,Z$.
	Hence the divisor $D\colonequals (0,1)-(\infty_+)$ represents a point $P
	\in J_0(39)(\Q)$ of finite order and $E\colonequals (0,-1)-(\infty_-)$
  represents $-P$ (in fact $\ord(P)=14$).
	Therefore the global $3$-adic height pairing $\langle D,E\rangle_\chi$
	must vanish, where $\chi$ is the cyclotomic character, say.
	Using {\url{https://github.com/emresertoz/neron-tate}}, we find that
	$\langle D,E\rangle_{\chi_v}=0$ for all primes $v\ne 3$. In this example, using our
	code, we will see that $\langle D,E\rangle_{\chi_p}=0$ as well.
	
	One can easily check that, by applying a suitable linear fractional transformation, the modular curve $X_0(39)$ can be transformed to a curve of the form
	\[y^2 = x(x-1)(x-r_0)(x-r_1)(x-r_2)(x-r_3)(x-r_4)\]
	where the branch points are in strong Kadziela position. 
  More precisely,
  the $r_i\in K$ satisfy
	\[v(r_0) = 3, \ \ \, v(r_1) = v(r_2) = v(r_3) = v(r_4) = 1.\]

  \begin{lemma}
    The curve $X_0(39)$ can be parameterized by a Whittaker group with fixed points $\{0,b_0,a_1,b_1,a_2,b_2,1,\infty\}$ that satisfy
		\[v(b_0) = 3, \ \ \, v(a_1) = v(b_1) = v(a_2) = v(b_2) = 1.\]
	\end{lemma}

	\begin{proof}
		By Corollary~\ref{cor:KadzielaShape}, we may assume that
		\[1 \leq v(b_2) \leq v(a_2) \leq v(b_1) \leq v(a_1) < v(b_0).\]
		Using Proposition~\ref{prop:v(F(z))>=v(z)}, we obtain $v(F(b_0)) \geq v(b_0) \geq 2$,
		which implies that $b_0$ is mapped to $r_0$ under $F$, because $r_0$ is the only root with valuation greater than $1$. Therefore, the other fixed points are mapped to $r_1,r_2,r_3,r_4$ in some order. Since all of these roots have valuation $1$, using Proposition~\ref{prop:v(F(z))>=v(z)} once again, we conclude that $v(a_1) = v(b_1) = v(a_2) = v(b_2) = 1$. 
		Finally, Lemma~\ref{lem:L1(z)_contribution} implies that $v(b_0) = 3$.
	\end{proof}
	
	Our code gives, with the help of this additional information, that the
  nontrivial fixed points are
	\begin{align*}
		b_0 &= \pi^{3} \cdot (169568969450082833675 \cdot i + 3 \cdot 81345423724756740290 \cdot \pi) + O(\pi^{90}), \\
		a_1 &= \pi \cdot (1817655172025484076360 \cdot i + 669944573840645064311 \cdot \pi) + O(\pi^{90}), \\
		b_1 &= \pi \cdot (1355295026970713034718 \cdot i + 433994741349790079549 \cdot \pi) + O(\pi^{90}), \\
		a_2 &= \pi \cdot (1818191821868583397124 \cdot i + 548477140915253215427 \cdot \pi) + O(\pi^{90}), \\
		b_2 &= \pi \cdot (2411648012235807762518 \cdot i + 95321642625024882380 \cdot \pi) + O(\pi^{90}).
	\end{align*}
	One can now easily check that these fixed points are in strong Kadziela position. The generators of the Whittaker group are
	\begin{small}
		\[
		w_1 = 
		\left(\begin{array}{rr}
			\pi^{3} \cdot (52793533972127 \cdot i + 3 \cdot 3103846656698 \cdot \pi) & \pi^{3} \cdot (31673686785512 \cdot i + 3 \cdot 1417904171591 \cdot \pi) \\
			2 & 83803859730842 + 3 \cdot 52793533972127 \cdot i \cdot \pi
		\end{array}\right) + O(\pi^{60}),
		\]
		\[
		w_2 = 
		\left(\begin{array}{rr}
			\pi \cdot (52144998882143 \cdot i + 66988365732322 \cdot \pi) & \pi \cdot (22184621499875 \cdot i + 30014311084481 \cdot \pi) \\
			2 & 200965097196962 + 52144998882143 \cdot i \cdot \pi
		\end{array}\right) + O(\pi^{60}),
		\]
		\[
		w_3 = 
		\left(\begin{array}{rr}
			\pi \cdot (62883996424702 \cdot i + 67395973180846 \cdot \pi) & \pi \cdot (11576205771028 \cdot i + 68215330456187 \cdot \pi) \\
			2 & 202187919542534 + 62883996424702 \cdot i \cdot \pi
		\end{array}\right) + O(\pi^{60}).
		\]
	\end{small}
	Moreover, the corresponding period matrix is
	\[
	Q = 
	\left(\begin{array}{rrr}
		3^{3} \cdot 2815951877525 & 3 \cdot 50001091700867 & 3 \cdot 646506176903 \\
		3 \cdot 50001091700867 & 3^{4} \cdot 1995748201645 & 3 \cdot 50001091700867 \\
		3 \cdot 646506176903 & 3 \cdot 50001091700867 & 3^{3} \cdot 2815951877525
	\end{array}\right)
	+ O(3^{30}),
	\]
	and the respective lifts of the divisors $D,E$ are
	\begin{align*}
		D' &= \Big(66967506124175 + 158782449129450 \cdot i + (187941388166341 + 96519270558011 \cdot i) \cdot \pi\Big) \\
		   & \ \ \ - \Big(\pi \cdot (187941388166341 + 109371861536638 \cdot i + (46307875323492 + 52927483043150 \cdot i) \cdot \pi)\Big) + O(\pi^{60}), \\
		E' &= \Big(\pi \cdot (17949743928308 + 109371861536638 \cdot i + (46307875323492 + 15702894321733 \cdot i) \cdot \pi)\Big) \\
		   & \ \ \ - \Big(66967506124175 + 47108682965199 \cdot i + (17949743928308 + 96519270558011 \cdot i) \cdot \pi\Big) + O(\pi^{60}).
	\end{align*}
	Finally, using Theorem~\ref{T:Werner}, our computations show that $\langle D,E\rangle_{\chi_p} = O(3^{30})$. The code for this example can be found in the file
	\texttt{X0(39).sage}.
	\begin{rk}\label{R:}
    In~\cite[\S9.1]{vdPT26}, the authors compute a first-order-approximation to a set
    of fixed points that are not in Kaziela position. 
  \end{rk}
	
\end{ex}

\bibliographystyle{alpha}
\bibliography{mumford}

\newcommand{\etalchar}[1]{$^{#1}$}
\begin{thebibliography}{{LMF}25}

\bibitem[AM19]{AM19}
Laia Amor\'os and Piermarco Milione.
\newblock Mumford curves covering {$p$}-adic {S}himura curves and their
  fundamental domains.
\newblock {\em Trans. Amer. Math. Soc.}, 371(2):1119--1149, 2019.

\bibitem[Bal15]{Bal15}
Jennifer~S. Balakrishnan.
\newblock Coleman integration for even-degree models of hyperelliptic curves.
\newblock {\em LMS J. Comput. Math.}, 18(1):258--265, 2015.

\bibitem[BB12]{BB12}
Jennifer~S. Balakrishnan and Amnon Besser.
\newblock Computing local $p$-adic height pairings on hyperelliptic curves.
\newblock {\em International Mathematics Research Notices},
  2012(11):2405--2444, 2012.

\bibitem[BB15]{BB15}
Jennifer~S. Balakrishnan and Amnon Besser.
\newblock Coleman-{G}ross height pairings and the {$p$}-adic sigma function.
\newblock {\em J. Reine Angew. Math.}, 698:89--104, 2015.

\bibitem[BBK10]{BBK10}
Jennifer~S. Balakrishnan, Robert~W. Bradshaw, and Kiran~S. Kedlaya.
\newblock Explicit {C}oleman integration for hyperelliptic curves.
\newblock In {\em Algorithmic number theory}, volume 6197 of {\em Lecture Notes
  in Comput. Sci.}, pages 16--31. Springer, Berlin, 2010.

\bibitem[BBM17]{BBM17}
Jennifer~S. Balakrishnan, Amnon Besser, and J.~Steffen M\"uller.
\newblock Computing integral points on hyperelliptic curves using quadratic
  {C}habauty.
\newblock {\em Math. Comp.}, 86(305):1403--1434, 2017.

\bibitem[BCP97]{Magma}
Wieb Bosma, John Cannon, and Catherine Playoust.
\newblock The {M}agma algebra system. {I}. {T}he user language.
\newblock {\em J. Symbolic Comput.}, 24(3-4):235--265, 1997.
\newblock Computational algebra and number theory (London, 1993).

\bibitem[BD18]{BD18}
Jennifer~S. Balakrishnan and Netan Dogra.
\newblock Quadratic {C}habauty and rational points, {I}: {$p$}-adic heights.
\newblock {\em Duke Math. J.}, 167(11):1981--2038, 2018.
\newblock With an appendix by J. Steffen M\"uller.

\bibitem[BD21]{BD21}
Matthew Bisatt and Tim Dokchitser.
\newblock Tame torsion and the tame inverse {G}alois problem.
\newblock {\em Math. Ann.}, 381(3-4):1439--1453, 2021.

\bibitem[BDM{\etalchar{+}}19]{BDMTV19}
Jennifer Balakrishnan, Netan Dogra, J.~Steffen M\"{u}ller, Jan Tuitman, and Jan
  Vonk.
\newblock Explicit {C}habauty-{K}im for the split {C}artan modular curve of
  level 13.
\newblock {\em Ann. of Math. (2)}, 189(3):885--944, 2019.

\bibitem[BDM{\etalchar{+}}23]{BDMTV23}
Jennifer~S. Balakrishnan, Netan Dogra, J.~Steffen M\"uller, Jan Tuitman, and
  Jan Vonk.
\newblock Quadratic {C}habauty for modular curves: algorithms and examples.
\newblock {\em Compos. Math.}, 159(6):1111--1152, 2023.

\bibitem[Bes21]{Bes21}
Alex~J. Best.
\newblock Square root time {C}oleman integration on superelliptic curves.
\newblock In {\em Arithmetic geometry, number theory, and computation}, Simons
  Symp., pages 105--129. Springer, Cham, [2021] \copyright 2021.

\bibitem[BKM25]{BKM25}
Francesca Bianchi, Enis Kaya, and J.~Steffen M\"uller.
\newblock Algorithms for {$p$}-adic heights on hyperelliptic curves of
  arbitrary reduction.
\newblock {\em Res. Number Theory}, 11(1):Paper No. 31, 21, 2025.

\bibitem[BKM26]{BKM}
Francesca Bianchi, Enis Kaya, and J.~Steffen Müller.
\newblock Coleman--{G}ross heights and $p$-adic {N}\'eron functions on
  {J}acobians of genus $2$ curves, 2026.
\newblock Preprint available at \url{https://arxiv.org/abs/2310.15049}.

\bibitem[BMS16]{BMS16}
Jennifer~S. Balakrishnan, J.~Steffen M\"{u}ller, and William~A. Stein.
\newblock A {$p$}-adic analogue of the conjecture of {B}irch and
  {S}winnerton-{D}yer for modular abelian varieties.
\newblock {\em Math. Comp.}, 85(298):983--1016, 2016.

\bibitem[BMS25]{BMS25}
Amnon Besser, J.~Steffen M\"uller, and Padmavathi Srinivasan.
\newblock p-adic adelic metrics and quadratic {C}habauty {I}.
\newblock {\em J. Reine Angew. Math.}, 828:223--305, 2025.

\bibitem[BT20]{BT20}
Jennifer~S. Balakrishnan and Jan Tuitman.
\newblock Explicit {C}oleman integration for curves.
\newblock {\em Math. Comp.}, 89(326):2965--2984, 2020.

\bibitem[CJ23]{CJ23}
Rudolf Chow and Frazer Jarvis.
\newblock A {$p$}-adic study of the {R}ichelot isogeny with applications to
  periods of certain genus 2 curves.
\newblock {\em Ramanujan J.}, 61(3):935--956, 2023.

\bibitem[CMR11]{CMR11}
Alan Carey, Matilde Marcolli, and Adam Rennie.
\newblock Modular index invariants of {M}umford curves.
\newblock In {\em Noncommutative geometry, arithmetic, and related topics},
  pages 31--73. Johns Hopkins Univ. Press, Baltimore, MD, 2011.

\bibitem[Col85a]{Col85Chab}
Robert~F. Coleman.
\newblock Effective {C}habauty.
\newblock {\em Duke Math. J.}, 52(3):765--770, 1985.

\bibitem[Col85b]{Col85}
Robert~F. Coleman.
\newblock Torsion points on curves and {$p$}-adic abelian integrals.
\newblock {\em Ann. of Math. (2)}, 121(1):111--168, 1985.

\bibitem[Gal96]{Gal96}
Steven~D. Galbraith.
\newblock {\em Equations for modular curves}.
\newblock PhD thesis, University of Oxford, 1996.

\bibitem[GM25]{GM25}
Stevan Gajovi\'c and J.~Steffen M\"uller.
\newblock Computing {$p$}-adic heights on hyperelliptic curves.
\newblock {\em Math. Comp.}, 94(354):2059--2088, 2025.

\bibitem[GvdP80]{GP80}
Lothar Gerritzen and Marius van~der Put.
\newblock {\em Schottky groups and {M}umford curves}, volume 817 of {\em
  Lecture Notes in Mathematics}.
\newblock Springer, Berlin, 1980.

\bibitem[Hol12]{Hol12}
David Holmes.
\newblock Computing {N}\'eron-{T}ate heights of points on hyperelliptic
  {J}acobians.
\newblock {\em J. Number Theory}, 132(6):1295--1305, 2012.

\bibitem[Kad07a]{KadPaper}
Samuel Kadziela.
\newblock Rigid analytic uniformization of curves and the study of isogenies.
\newblock {\em Acta Appl. Math.}, 99(2):185--204, 2007.

\bibitem[Kad07b]{Kad07}
Samuel Kadziela.
\newblock {\em Rigid analytic uniformization of hyperelliptic curves}.
\newblock ProQuest LLC, Ann Arbor, MI, 2007.
\newblock Thesis (Ph.D.)--University of Illinois at Urbana-Champaign.

\bibitem[Kay22]{Kay22}
Enis Kaya.
\newblock Explicit {V}ologodsky integration for hyperelliptic curves.
\newblock {\em Math. Comp.}, 91(337):2367--2396, 2022.

\bibitem[Kim05]{kim05:motivic_fundamental_group}
Minhyong Kim.
\newblock The motivic fundamental group of {$\BP^1\setminus\{0,1,\infty\}$} and
  the theorem of {S}iegel.
\newblock {\em Invent. Math.}, 161(3):629--656, 2005.

\bibitem[Kim09]{kim09:unipotent_albanese}
Minhyong Kim.
\newblock The unipotent {A}lbanese map and {S}elmer varieties for curves.
\newblock {\em Publ. Res. Inst. Math. Sci.}, 45(1):89--133, 2009.

\bibitem[KK22]{KK22}
Eric Katz and Enis Kaya.
\newblock {$p$}-adic integration on bad reduction hyperelliptic curves.
\newblock {\em Int. Math. Res. Not. IMRN}, (8):6038--6106, 2022.

\bibitem[KRZB16]{KRBZ16}
Eric Katz, Joseph Rabinoff, and David Zureick-Brown.
\newblock Uniform bounds for the number of rational points on curves of small
  {M}ordell-{W}eil rank.
\newblock {\em Duke Math. J.}, 165(16):3189--3240, 2016.

\bibitem[Lan88]{Lan88}
Serge Lang.
\newblock {\em Introduction to {A}rakelov theory}.
\newblock Springer-Verlag, New York, 1988.

\bibitem[{LMF}25]{LMFDB}
The {LMFDB Collaboration}.
\newblock The {L}-functions and modular forms database.
\newblock \url{https://www.lmfdb.org}, 2025.
\newblock [Online; accessed 16 December 2025].

\bibitem[MD73]{MD73}
Yuri Manin and Vladimir Drinfeld.
\newblock Periods of {$p$}-adic {S}chottky groups.
\newblock {\em J. Reine Angew. Math.}, 262/263:239--247, 1973.

\bibitem[MP12]{MP12}
William McCallum and Bjorn Poonen.
\newblock The method of {C}habauty and {C}oleman.
\newblock In {\em Explicit methods in number theory}, volume~36 of {\em Panor.
  Synth\`eses}, pages 99--117. Soc. Math. France, Paris, 2012.

\bibitem[MR15]{MR15}
Ralph Morrison and Qingchun Ren.
\newblock Algorithms for {M}umford curves.
\newblock {\em J. Symbolic Comput.}, 68(part 2):259--284, 2015.

\bibitem[MT83]{MT83}
Barry Mazur and John Tate.
\newblock Canonical height pairings via biextensions.
\newblock In {\em Arithmetic and geometry, {V}ol. {I}}, volume~35 of {\em
  Progr. Math.}, pages 195--237. Birkh\"auser Boston, Boston, MA, 1983.

\bibitem[MTT86]{mtt}
Barry Mazur, John Tate, and Jeremy Teitelbaum.
\newblock On {$p$}-adic analogues of the conjectures of {B}irch and
  {S}winnerton-{D}yer.
\newblock {\em Invent. Math.}, 84(1):1--48, 1986.

\bibitem[M{\"u}l14]{Mul14}
J.~Steffen M{\"u}ller.
\newblock Computing canonical heights using arithmetic intersection theory.
\newblock {\em Math. Comp.}, 83(285):311--336, 2014.

\bibitem[Mum72]{Mu72}
David Mumford.
\newblock An analytic construction of degenerating curves over complete local
  rings.
\newblock {\em Compositio Math.}, 24:129--174, 1972.

\bibitem[MX26]{MX25}
Marc Masdeu and Xavier Xarles.
\newblock Efficient computation of non-archimedean theta functions.
\newblock {\em Math. Comp.}, 95(357):457--475, 2026.

\bibitem[Sch82]{Sch82}
Peter Schneider.
\newblock {$p$}-adic height pairings. {I}.
\newblock {\em Invent. Math.}, 69(3):401--409, 1982.

\bibitem[Sto19]{Sto19}
Michael Stoll.
\newblock Uniform bounds for the number of rational points on hyperelliptic
  curves of small {M}ordell-{W}eil rank.
\newblock {\em J. Eur. Math. Soc. (JEMS)}, 21(3):923--956, 2019.

\bibitem[SW13]{SW13}
William Stein and Christian Wuthrich.
\newblock Algorithms for the arithmetic of elliptic curves using {I}wasawa
  theory.
\newblock {\em Math. Comp.}, 82(283):1757--1792, 2013.

\bibitem[Tei88]{Tei88}
Jeremy Teitelbaum.
\newblock {$p$}-adic periods of genus two {M}umford-{S}chottky curves.
\newblock {\em J. Reine Angew. Math.}, 385:117--151, 1988.

\bibitem[{The}26]{sagemath}
{The Sage Developers}.
\newblock {\em {S}ageMath, the {S}age {M}athematics {S}oftware {S}ystem
  ({V}ersion 10.8)}, 2026.
\newblock {\tt https://www.sagemath.org}.

\bibitem[vBHM20]{vBHM20}
Raymond van Bommel, David Holmes, and J.~Steffen M\"{u}ller.
\newblock Explicit arithmetic intersection theory and computation of
  {N}\'{e}ron-{T}ate heights.
\newblock {\em Math. Comp.}, 89(321):395--410, 2020.

\bibitem[vdP79]{vdP78}
Marius van~der Put.
\newblock $p$-adic {W}hittaker groups.
\newblock {\em Groupe {\'e}tude Anal. Ultram{\'e}tr.}, 6(15):1--6, 1978-79.

\bibitem[vdPT26]{vdPT26}
Marius van~der Put and Jaap Top.
\newblock Whittaker groups and hyperelliptic curves, 2026.
\newblock Preprint available at \url{https://arxiv.org/abs/2605.22406}.

\bibitem[vS81]{VanSteen}
Guido van Steen.
\newblock {\em Hyperelliptic curves defined by Schottky groups over a
  non-archimedean valued field}.
\newblock PhD thesis, University of Antwerp, 1981.

\bibitem[Wer96]{Wer96}
Annette Werner.
\newblock Local heights on {M}umford curves.
\newblock {\em Math. Ann.}, 306(4):819--831, 1996.

\bibitem[Wer97]{Wer97}
Annette Werner.
\newblock Local heights on abelian varieties with split multiplicative
  reduction.
\newblock {\em Compositio Math.}, 107(3):289--317, 1997.

\bibitem[Wer98]{Wer98}
Annette Werner.
\newblock Local heights on abelian varieties and rigid analytic uniformization.
\newblock {\em Doc. Math.}, 3:301--319, 1998.

\bibitem[Yel24a]{Yel23}
Jeffrey Yelton.
\newblock Branch points of split degenerate superelliptic curves {I}:
  construction of {S}chottky groups, 2024.
\newblock Preprint available at \url{https://arxiv.org/abs/2306.17823}.

\bibitem[Yel24b]{Yel24}
Jeffrey Yelton.
\newblock Branch points of split degenerate superelliptic curves {II}: on a
  conjecture of {G}erritzen and van der {P}ut, 2024.
\newblock Preprint available at \url{https://arxiv.org/abs/2407.11303}.

\bibitem[Zar96]{Zar96}
Yuri~G. Zarhin.
\newblock {$p$}-adic abelian integrals and commutative {L}ie groups.
\newblock volume~81, pages 2744--2750. 1996.
\newblock Algebraic geometry, 4.

\end{thebibliography}

\end{document}